\documentclass[12pt]{amsart}
\usepackage{graphicx, fullpage}
\usepackage{amsmath,amssymb,amsthm,amscd, bm}
\usepackage{fancyhdr}
\usepackage[mathscr]{eucal}
\usepackage{amsfonts}
\usepackage[T1]{fontenc}
\usepackage{xspace}
\usepackage{esint}
\theoremstyle{plain}
\newtheorem {thm}{Theorem}[section]
\newtheorem {lem}[thm]{Lemma}
\newtheorem {cor}[thm]{Corollary}
\newtheorem {defn}[thm]{Definition}
\newtheorem {prop}[thm]{Proposition}
\numberwithin{equation}{section}

\def\cal{\mathcal}
\newcommand{\ddb}{\partial\bar\partial}

\newcommand{\cO}{{\cal O}}
\newcommand{\cE}{{\cal E}}

\newcommand{\cK}{{\cal K}}

\newcommand{\cH}{{\cal H}}
\newcommand{\cL}{{\cal L}}

\newcommand{\cV}{{\cal V}}

\newcommand{\fg}{\mathfrak g}
\newcommand{\fh}{\mathfrak h}
\newcommand{\fu}{\mathfrak{u}}

\newcommand{\bC}{\mathbb C}
\newcommand{\bP}{\mathbb P}
\newcommand{\bN}{\mathbb N}

\newcommand{\bR}{\mathbb R}

\newcommand{\one}{\mathcal O_{\bP F}(1)}
\newcommand{\onb}{\mathcal O_{\bP F_b}(1)}

\newcommand{\tr}{\text{tr}}

\newcommand{\var}{\varepsilon}

\newcommand{\GL}{\text{GL}}
\newcommand{\SU}{\text{SU}}
\newcommand{\SL}{\text{SL}}

\newcommand{\dist}{\text{dist}}

\newcommand{\id}{\text{id}}
\newcommand{\po}{\text{Psh}(\omega)}

\newcommand{\re}{\text{Re}\,}

\newcommand{\sg}{\text{sgrad}\,}
\newcommand{\mx}{\langle\mu,X\rangle}

\newcommand{\supp}{\text{supp\,}}
\newcommand{\vo}{\text{Vol}}

\begin{document}
\title{\bf Two variational problems in K\"ahler geometry}
\author{L\'aszl\'o Lempert}
\address{Department of  Mathematics,
Purdue University, 150N University Street, West Lafayette, IN
47907-2067, USA}
\subjclass[2020]{32Q15, 32U05, 32W20, 52A40, 58E40}

%\thispagestyle{empty}
%\end{titlepage}
\abstract
On a K\"ahler manifold $(M,\omega)$ we consider the problems of maximizing/minimizing 
Monge--Amp\`ere energy
over certain subsets of the space of K\"ahler potentials. Under suitable assumptions we prove that
solutions to these
variational problems exist, are unique, and have a simple characterization. We then use the extremals to
construct hermitian metrics on holomorphic vector bundles, and investigate their curvature.
 
\endabstract
\maketitle
\section{Introduction}    % Section 1

The pair of variational problems studied in this paper goes back to two sources. One has to do with
constructing hermitian metrics on holomorphic vector bundles; we will discuss this aspect in section 8.
The other source is certain problems in convex geometry: For example, given a
bounded set $S\subset\bR^m$, among ellipsoids centered at $0$ and containing $S$ characterize the one
of least volume. Alternatively, volume minimization can be considered without prescribing the center. This latter 
problem was posed and solved by John in the influential paper \cite{J48}, see also \cite{B97}. The result has 
also been associated with L\"owner's name. There is a dual problem as well, maximizing the volume 
of ellipsoids  inscribed in a convex set; either problem can be easily reduced to the other.

The problems have natural complex variants. For example, consider a convex $S\subset\bC^m$,
and among all hermitian ellipsoids---images of the unit ball under complex linear transformations---that
are contained in $S$ search for the one with maximal volume. This is still a convex programing 
problem, and Arocha, Bracho, and Montejano in \cite{A+22} find that John's method will provide a
solution. Related earlier work includes Gromov's, who in \cite{G67} already dealt with the problem when
$S$ is symmetric about the origin, and \cite{R86,M97} by Rabotin and Ma. 
It is noteworthy, though, that a variant of the
approach of John and Arocha et al. gives that the complex analog of John's result holds under a suitable
{\sl pseudo}convexity assumption on $S$, more natural in complex analysis than geometric convexity.

A new perspective is gained if we blow up $\bC^m$ at the origin. The blow up is the total space of the
tautological line bundle $L$ over complex projective space $\bP_{m-1}$. If $S$ was a 
balanced\footnote{This means $\lambda S\subset S$ for $\lambda\in\bC$, $|\lambda|\le 1$.} 
neighborhood of $0$,
its blow up will be the unit disc bundle of a hermitian metric $h$ on $L$. There are several ways 
to characterize $h$ that arise from hermitian ellipsoids $S$. These are metrics that can be obtained as the
dual of the Fubini--Study metric on $L^*$, composed with arbitrary $g\in\GL(m,\bC)$ acting on $L$.
Another characterization takes $-i$ times the curvature (the negative of the Chern form) of $h$, a K\"ahler form
$\omega$ on $\bP_{m-1}$; the requirement is that the scalar curvature of the metric determined by 
$\omega$ should be constant (equivalently, $\omega$ should be K\"ahler--Einstein).

This suggests a first generalization, a variational problem involving hermitian metrics on a holomorphic
line bundle $L\to M$ over a projective manifold. We will not spell out the problem at this point, because a further
generalization simplifies  matters by abandoning line bundles and hermitian metrics altogether. Instead,
consider a connected compact K\"ahler manifold $(M,\omega)$; the role of hermitian
metrics will be played by (relative) K\"ahler potentials, that form the space
\begin{equation} %1.1
\cH=\cH(\omega)=\{u\in C^\infty(M):\omega_u=\omega+i\ddb u>0\}.
\end{equation}
Here $C^\infty(M)$ stands for the Fr\'echet space of smooth functions $M\to\bR$.

Suppose a connected compact Lie group $G$ acts on the left on $(M,\omega)$ by 
holomorphic isometries. The
action extends to a holomorphic action of the complexified group $G^\bC$
\[
G^\bC\times M\ni(g,x)\mapsto gx\in M.
\]
The induced action of $g$ on differential forms will be denoted $g^*$.
\begin{defn} %1.1
A K\"ahler potential $u\in\cH$ is admissible if there is a $g\in G^\bC$ such that $\omega_u=g^*\omega$.
\end{defn}

Since $\omega$ and $g^*\omega$ are in the same de Rham cohomology class, 
Hodge theory gives that for every $g\in G^\bC$ there 
is a $u$ as in the definition; moreover, $g$ determines $u$ up to an additive constant. If $g\in G$, $u$ is
constant.

Given a bounded $v_0:M\to\bR$, inscribed/circumscribed ellipsoids of the convex world correspond in
our setting to admissible $v\le v_0$, respectively $v\ge v_0$. It is not hard to check that the 
logarithm of the volume of an ellipsoid corresponds to Monge--Amp\`ere energy\footnote{The notion,
along with other background material, will be reviewed in section 2.} 
$E(v)$.  (A related result is Lemma 8.2.) Accordingly, the variational problems
to be studied are
\begin{align}
\max\{E(v):v\in\cH\text{ is admissible, }v\le v_0  \}, \\ %1.2
\min\{E(v):v\in\cH \text{ is admissible, }v\ge v_0\}. %1.3
\end{align}

Under suitable assumptions we will show that solutions exist, are unique, and can be characterized
by a condition involving the moment map of the $G$-action, analogous to John's condition. Let $\fg$
denote the Lie algebra of $G$, $\fg^*$ its dual, and $\langle\,,\rangle$ the pairing between
$\fg^*$ and $\fg$. The (equivariant) moment map is a smooth
map $\mu:M\to \fg^*$ such that for every $X\in\fg$ the function
\[
\langle\mu,X\rangle\in C^\infty(M)
\]
is a Hamiltonian of the vector field $X_M$ on $M$ that corresponds to $X$ under the action of $G$. 
This latter condition determines
$\langle\mu, X\rangle$ up to an additive constant, that we choose so that 
$\int_M\langle\mu,X\rangle\omega^n=0$. Not every action has a moment map; we assume ours has.

The following property of $\mu$ will be relevant for uniqueness. Let $N_X\subset M$
consist of points where $\langle\mu, X\rangle$ vanishes to second order (at least), and consider the
condition:
\begin{equation} %1.4
\text{If $X\in\fg$ is nonzero, the convex hull of $\{\mu(x):x\in N_X\}\subset\fg^*$ avoids $0\in\fg^*$.} 
\end{equation}
Again, this property is not guaranteed. In section 5 we will see examples where it holds and where it does
not hold.

To formulate our theorems, for $u\in\cH$ define the contact set $C_u\subset M$ by
\begin{equation} %1.5
C_u=\{x\in M: u(x)=v_0(x)\}.
\end{equation}
\begin{thm} %1.2
Given a continuous $v_0:M\to\bR$, the maximum in (1.2) is attained. An admissible $u\le v_0$, with
$\omega_u=g^*\omega$, $g\in G^\bC$, maximizes
$E(v)$ in (1.2) if and only if
\begin{equation} %1.6
\text{the convex hull of $\{\mu(gx): x\in C_u\}\subset \fg^*$ contains 0.}
\end{equation}
Furthermore, if $\mu$ satisfies (1.4), the maximizer in (1.2) is unique.
\end{thm}

Theorem 1.2 puts no restriction, reminiscent of convexity, on $v_0$. This is natural, since problem
(1.2) does not change if $v_0$ is replaced by its ``$\omega$--plurisubharmonic envelope''. By contrast,  
in problem (1.3) $\omega$--plurisubharmonicity will be important.

\begin{thm} %1.3
Given a continuous $v_0:M\to\bR$, the minimum in (1.3) is attained. Assuming $\omega+i\ddb v_0\ge 0$
as a current, an admissible $u\ge v_0$, with $\omega_u=g^*\omega$, $g\in G^\bC$, minimizes
$E(v)$ in (1.3) if and only if (1.6) holds. Finally, if $\mu$ satisfies (1.4), $v_0\in C^2(M)$, and 
$\omega+i\ddb v_0>0$, then the minimizer in (1.3) is unique.
\end{thm}

In section 5 we will see that without assuming (1.4), in both problems uniqueness may fail. Also, it 
is not hard to construct examples where (1.4) holds but $\omega+i\ddb v_0$ fails to be positive, and the 
solution of problem (1.3) is not unique.

Above we mentioned in passing that hermitian ellipsoids in $\bC^m$ correspond to K\"ahler metrics
on $\bP_{m-1}$ of constant scalar curvature. This suggests different variational problems on
general K\"ahler manifolds $(M,\omega)$: Given $v_0:M\to\bR$, maximize/minimize $E(v)$ among
constant scalar curvature potentials $v\le v_0$, respectively $v\ge v_0$ (i.e., the metric of $\omega_v$
should have constant scalar curvature). Of course, we have to assume $\cH(\omega)$ contains 
constant scalar curvature potentials. If so, no generality is lost by assuming that the metric
of $\omega$ itself has constant scalar curvature. This pair of variational problems fits with (1.2), (1.3). Indeed,
if $v$ is a potential of constant scalar curvature, then by a result of Berman and Berndtsson 
it is admissible for the action of the ``reduced'' isometry group $G$ of $(M,\omega)$, see 
Theorem 4.8 and the comment following it in \cite{BB17}. The converse is also true, rather obviously. (The
reduced isometry group is the connected subgroup of the group of holomorphic isometries with Lie algebra consisting
of Killing vector fields that vanish somewhere.) By results of Lichnerowicz and Kobayashi, \cite{K72, L67},
to which we will come shortly, such actions always admit a moment map. In conclusion, the 
variational problems for constant scalar curvature potentials are special cases of (1.2), (1.3).

If ellipsoids and K\"ahler metrics of constant scalar curvature are analogous, what corresponds in the
convex world to the general variational problems (1.2), (1.3), when $\omega$ is arbitrary? A first
guess would be for two convex bodies $S,T\subset\bC^m$ that are, say, symmetric about the origin, 
finding the linear image of $S$ that is contained in $T$ and has maximal volume among such images. This
problem was solved by V. Milman, see \cite[Theorem 14.5]{T89}
and the comments on p. 137 there; the more
general problem without symmetry was treated in \cite{G+01}. The analogy, however, is imperfect. While
in \cite[Theorem 14.5]{T89} all linear images are allowed, the spirit of our theorems would only allow images 
under $g$ in the complexification of a subgroup $G\subset\GL(\bC,m)$ that leaves $S$ invariant.
This is the reason why Theorems 1.2, 1.3 can claim uniqueness, whereas in the convex geometry
problems uniqueness fails sometimes.
Nonetheless, the criteria for the extremals in \cite[Theorem 14.5]{T89} and in \cite{G+01} on the one hand,
and in Theorems 1.2, 1.3 above on the other, are quite similar.

Problems (1.2), (1.3) suggest certain ways to construct hermitian metrics on holomorphic vector bundles.
In later sections we investigate one of these constructions, and curvature properties of the metrics obtained
through it.

Contents. In section 2 we review matters concerning plurisubharmonicity in K\"ahler
manifolds, geodesics in Mabuchi's space $\cH$ and beyond, Monge--Amp\`ere energy, the moment map, 
and complexified group actions. Sections 3,4, and 5 prove various parts of Theorems 1.2,
1.3: section 3 the existence of extremals, section 4 their characterization, and section 5 their uniqueness,
occasionally under less restrictive assumptions than what is formulated in the theorems above. 
Section 6
proves stability of the extremals and extrema under perturbation of the barrier $v_0$. In section 7 we do spell
out the variational problems for hermitian metrics on line bundles corresponding to (1.2), (1.3)---something we
abandoned earlier---and translate, in part, Theorems 1.2, 1.3 to this setting. Section 8 then introduces a
construction of hermitian metrics on vector bundles, motivated by problem (1.3), and sections 9, 10 
investigate the curvature of the hermitian metrics constructed.

During the preparation of this paper I have profited from discussions with Kuang--Ru Wu, and from his comments on earlier
versions.

\section{Background} %2

This section reviews notions that will be needed further down.

{\bf Spaces of plurisubharmonic functions.}  Suppose $N$ is a complex manifold and $\Omega$ a smooth,
real $(1,1)$ form, $d\Omega=0$. A function $f:N\to[-\infty,\infty)$ is called $\Omega$-plurisubharmonic if
on every open $U\subset N$ on which $\Omega$ can be written as $i\ddb w$ with $w:U\to\bR$ smooth, 
the function
$w+f$ is plurisubharmonic. The collection of such functions is denoted Psh$(\Omega)$.
For continuous functions $f$, $\Omega$-plurisubharmonicity is equivalent to $\Omega+i\ddb f\ge 0$
as a current.

Suppose $(M,\omega)$ is an $n$ dimensional, connected compact K\"ahler manifold. The space
$\text{Psh}(\omega)$ and its subspaces (such as $\cH$ in (1.1)) have a rich geometry, see 
\cite{D15,D17,D19}. A central notion
is geodesics. After Berndtsson and Darvas, they are defined as follows \cite{B18, D19}. 
Let $a<b$ be real numbers, $S_{ab}=\{s\in\bC: a<\re s<b\}$, and $\pi: S_{ab}\times M\to M$ the projection.
\begin{defn} %2.1
(a) A path (i.e., a continuous map) $\varphi:(a,b)\to \text{Psh}(\omega)$ is a subgeodesic if $\Phi:S\times M\to[-\infty,\infty)$ defined
by $\Phi(s,x)=\varphi(\re s)(x)$ is $\pi^*\omega$-plurisubharmonic.

(b) The geodesic $\varphi:(a,b)\to\text{Psh}(\omega)$ connecting $u,v\in\text{Psh}(\omega)$ is
\begin{equation} %2.1
\varphi=\sup\{\psi\mid\psi:(a,b)\to\text{Psh}(\omega)\text{ is a subgeodesic, }\lim_a\psi\le u,\lim_b\psi\le v\}.
\end{equation}
\end{defn}
In (2.1) the pointwise limits exist because for subgeodesics $\psi(\cdot)(x)$ are convex functions for all
$x\in M$. Often one extends geodesics by setting $\varphi(a)=u$, $\varphi(b)=v$, and one still calls $\varphi$
thus extended a geodesic. A path $\varphi:\bR\to\po$ is called  geodesic if its restriction to any
$[a,b]\subset\bR$ is.

The geodesics of relevance to this paper are all smooth maps into $\cH\subset\text{Psh}(\omega)$. 
This class of geodesics had been introduced by Mabuchi in \cite{M87} rather differently: 
$\cH\subset C^\infty(M)$, as an open subset of a Fr\'echet space, inherits the structure of a Fr\'echet
manifold, whose tangent bundle has a canonical trivialization $T\cH\approx\cH\times C^\infty(M)$.
Mabuchi defines a Riemannian metric $g_M$ on $\cH$ by letting $\vo=\int_M\omega^n$ $(=\int_M\omega_u^n$
for all $u\in\cH$) and
\[
g_M(\xi,\eta)=\dfrac1\vo\int_M\xi\eta\,\omega_u^n,\qquad \xi,\eta\in T_u\cH\approx C^\infty(M).
\]
The metric determines a unique Levi--Civita connection $\nabla$ on $T\cH$. According to Mabuchi,
a smooth $\varphi:[a,b]\to\cH$ is a geodesic if its velocity $\dot\varphi:[a,b]\to T\cH$ is parallel along $\varphi$,
$\nabla_{\dot\varphi(t)}\dot\varphi (t)=0$, $a\le t\le b$.

That Mabuchi's geodesics are geodesics in the sense of Definition 2.1 follows from Semmes' work 
\cite{S92}. Semmes rediscovered Mabuchi's connection, and proved that a smooth $\varphi:[a,b]\to\cH$
is a geodesic for this connection if and only if the function
\begin{equation} %2.2
\Phi: \bar S_{ab}\times M\ni(s,x)\mapsto\varphi(\re s)(x)\in\bR
\end{equation}
satisfies the Monge--Amp\`ere equation
\begin{equation} %2.3
(\pi^*\omega+i\ddb\Phi)^{n+1}=0,\qquad\text{or, equivalently,}\qquad \text{rk}\,(\pi^*\omega+i\ddb\Phi)\equiv n.
\end{equation}
(2.3) implies that $\Phi\in\text{Psh}(\pi^*\omega)$, i.e., $\varphi$ is a subgeodesic, and it follows from
\cite{BK77} by Bedford and Kalka that $\bar S_{ab}\times M$ is foliated by graphs of smooth functions
$f:\bar S_{ab}\to M$, holomorphic on $S_{ab}$, along which $\pi^*\omega+i\ddb\Phi$ restricts to $0$.
Setting $u=\varphi(a)$, $v=\varphi(b)$, the maximum principle, applied on the leaves,
then yields that $\varphi$ dominates
all subgeodesics $\psi$ in (2.1): it is therefore the geodesic connecting $u,v$.

\begin{prop} %2.2
If $v_0:M\to\bR$ is any function, the set $A=\{v\in\text{Psh}(\omega): v\le v_0\}$ is convex in the sense
that if the endpoints of a geodesic $\varphi:[a,b]\to\po$ are in $A$, then $\varphi(t)\in A$ for all $t$.
Similarly, if $v_0\in\po$, then $B=\{v\in\po: v\ge v_0\}$ is convex.
\end{prop}
\begin{proof} Suppose a geodesic, or even a subgeodesic $\varphi:[a,b]\to\po$ has endpoints in $A$.
Then the associated function $\Phi:\bar S_{ab}\times M\to [-\infty,\infty)$ as in (2.2) is $\pi^*\omega$-plurisubharmonic 
on $S_{ab}\times M$. Since it is translation invariant along the imaginary direction in
$S_{ab}$, for each $x\in M$ the function $\varphi(\cdot)(x)$ is convex, or identically $-\infty$ on $(a,b)$. Hence
$\varphi(t)(x)\le v_0(x)$ for $t=a,b$ implies the same for all $t\in[a,b]$.

Next suppose our geodesic $\varphi$ has endpoints in $B$. Now $\psi(t)=v_0$ for $t\in[a,b]$ is
a subgeodesic and
$\psi(t)\le\varphi(t)$ for $t=a,b$. Hence Definition 2.1 gives $\psi(t)\le\varphi(t)$ for all $t$, i.e., $\varphi(t)\in B$.
\end{proof}
\begin{prop} %2.3
Riemannian distance between $u,v\in\cH$ is $\le\max_M|u-v|$.
\end{prop}
\begin{proof} Indeed, the length of the path $\phi(t)=u+t(v-u)$, $0\le t\le 1$, connecting $u,v$ is
$\le\max_M|u-v|$.
\end{proof}
Mabuchi further introduced a closed smooth 1-form, say  $\alpha$, on $\cH$,
\[
\alpha(\xi)=\dfrac1\vo\int_M\xi\,\omega_u^n,\qquad \xi\in T_u\cH\approx C^\infty(M).
\]
$\cH\subset C^\infty(M)$ being convex, $\alpha$ must be exact. Monge--Amp\`ere energy
$E:\cH\to\bR$ is defined as a primitive of $\alpha$,
\begin{equation} %2.4
dE=\alpha,\qquad E(0)=0.
\end{equation} 
The zero level set $\cK=E^{-1}(0)$ is totally geodesic, and $u\mapsto \omega_u$ is a bijection between
$\cK$ and the set of K\"ahler forms $d$--cohomologous to $\omega$. 
\begin{prop} %2.4
The map
\[
\cH\ni u\mapsto \big(u-E(u), E(u)\big)\in\cK\times\bR
\]
is a Riemannian isometry if $\bR$ is endowed with the Euclidean metric (see \cite[top of p. 233]{M87}). In
particular, $E$ is linear along geodesics in $\cH$.
\end{prop}

The only geodesics needed in the paper are of the following type. Let $\cL$ stand for Lie derivative, and
let $J:TM\to TM$ be the almost complex structure tensor of $M$.
\begin{prop} %2.5
Suppose a real vector field $Y$ on $M$ vanishes somewhere and generates a one parameter group 
$g_t$ of biholomorphisms. If
$\cL_{JY}\omega=0$, $h\in C^\infty(M)$ satisfies $i\ddb h=\cL_{Y}\omega$ and $\int_M h\,\omega^n=0$,
then with any $c\in\bR$
\begin{equation} %2.5
\varphi(t)=\int_0^tg_\tau^*h\, d\tau\qquad\text{and}\qquad \psi(t)=\varphi(t)+ct
\end{equation}
define geodesics $\varphi:\bR\to\cK$, $\psi:\bR\to\cH$.
\end{prop}
\begin{proof}
First observe that
\begin{equation} %2.6
i\ddb\psi(t)=i\ddb\varphi(t)=\int_0^t g_\tau^*i\ddb h\,d\tau=\int_0^tg_\tau^*\cL_{Y}\omega\,d\tau=
g_t^*\omega-\omega.
\end{equation}
The differential of Monge--Amp\`ere energy evaluated on the velocity vector $\dot\varphi(t)$ is
\[
dE\big(\dot\varphi(t)\big)=\dfrac1\vo\int_M g_t^*h\,\omega_{\varphi(t)}^n=\dfrac1\vo\int_Mg_t^*(h\omega^n)=
\dfrac1\vo\int_Mh\,\omega^n=0.
\]
Hence $E\big(\varphi(t)\big)=E\big(\varphi(0)\big)=0$ for all $t$, and $\varphi(t)\in\cK$. Mabuchi 
and Donaldson already 
studied paths in $\cK$ induced by vector fields on $M$ that generate one parameter subgroups of
biholomorphisms. \cite[Theorem 3.5]{M87} implies that $\varphi$ is a geodesic (for a different 
perspective, see \cite[\S 6]{D99}, and for a generalization, Theorem 9.7 here). Therefore $\psi$ is also
geodesic by Proposition 2.4.
\end{proof}

{\bf The moment map.} Now let $(M,\omega)$ be a connected compact symplectic manifold of dimension
$2n$. Contraction with tangent vectors 
\begin{equation} %2.7
TM\ni\theta\mapsto-\iota(\theta)\omega\in T^*M,\quad\text{resp.}\quad 
\bC TM\ni\theta\mapsto-\iota(\theta)\omega\in\bC T^*M,
\end{equation}
induces isomorphisms between the tangent 
and cotangent bundles of $M$ on the one hand, and their complexifications on the other. The 
symplectic gradient of a smooth $h:M\to\bC$ is the complex vector field $\sg h=\text{sgrad}_\omega \,h$ 
that corresponds to
$dh$ under (2.7). Vector fields that arise in this way are called Hamiltonian. If $h$ is real valued,
$\sg h$ is a real vector field whose flow preserves $\omega$.

Consider a smooth left action of a Lie group $G$ on $M$
\begin{equation} %2.8
G\times M\ni(g,x)\mapsto gx\in M.
\end{equation}
The differential of the above map at $\{e\}\times M$ induces a homomorphism from $\fg$, the Lie
algebra of $G$, to Vect$(M)$,
the Lie algebra  of smooth real vector fields on $M$. Denote the vector field that corresponds to
$X\in\fg$ by $X_M$. If $G$ acts by symplectomorphisms, the flow of these vector fields preserves
$\omega$, $0=\cL_{X_M}\omega=d\iota(X_M)\omega$. Assuming that $\iota(X_M)\omega$ is not
only closed but exact, we can write $-\iota(X_M)\omega=dh_X$, or $X_M=\sg h_X$ with some
$h_X\in C^\infty(M)$, that is unique under the condition $\int_Mh_X\,\omega^n=0$. If such $h_X$
exists for all $X\in\fg$, the moment map $\mu:M\to\fg^*$ is defined by
\[
\langle\mu(x), X\rangle=h_X(x),\qquad x\in M, \quad X\in\fg;
\]
it is equivariant in the sense that $\mu\circ g=\text{Ad}^*_{g^{-1}}\,\mu$, see e.g. \cite[pp. 186--187]{GS84'}.

Assume now that $(M,\omega)$ is even K\"ahler, and the action of $G$ preserves the complex 
structure in addition to
$\omega$. Then in fact each $g\in G$ acts by a holomorphic isometry, and there is a good criterion,
due to Lichnerowicz and Kobayashi, for the existence of the moment map. Denote by
$J$ the complex structure tensor of $M$.  Any complex vector field $Y$
on $M$ can be decomposed $Y=Y'+Y''$ into $(1,0)$ and $(0,1)$ vector fields,
\begin{equation} %2.9
Y'=\frac12(Y-iJY),\quad Y''=\frac12(Y+iJY);\quad\text{if } Y\text{ is real, } Y=2\re Y'=2\re Y''.
\end{equation}
That the vector fields $X_M$ generate one parameter groups of holomorphic isometries means that
$X_M$ is a Killing vector field, and $X_M'$ is a holomorphic section of the $(1,0)$ tangent bundle
$T^{10}M\to M$ \cite[p. 77]{K72}.
\begin{prop} %2.6
A holomorphic $(1,0)$ vector field $Z$ on $M$ can be written $Z=(\sg h)'$ with some smooth $h:M\to\bC$
if and only if $\sigma(Z)=0$ for all holomorphic $(1,0)$-forms $\sigma$ on $M$. $\re Z$ is Killing if and
only if $h$ can be chosen real; then $\re Z=\sg h/2$.
\end{prop}
\begin{proof} 
See \cite{L67}, \cite[pp. 94--95]{K72}. In Kobayashi's notation the $(0,1)$-form $\zeta$ 
corresponding to the vector field $Z$ is an imaginary constant times $\iota(Z)\omega$.
\end{proof}

Since $\sigma(Z)$ above is a holomorphic function on $M$, hence constant, it follows that $\sigma(Z)=0$
whenever $Z$ (or $\re\,Z$) vanishes somewhere.  Conversely, if $\re Z=\sg h$ with a real $h$, then
$Z=0$ where $h$ is maximal. (A fuller result of this sort is \cite[Theorem 1]{LS94} by LeBrun and Simanca.) Hence:
\begin{cor} %2.7
Suppose a connected Lie group $G$ acts by holomorphic isometries on a connected compact K\"ahler
manifold $M$. The action admits a moment map if and only if the vector fields $X_M$ vanish somewhere,
for all $X\in\fg$ or,
equivalently, for a family of $X\in\fg$ that spans $\fg$.
\end{cor}

For example, $\SU(n+1)$  acts naturally by holomorphic isometries on $M=\bP_n$, endowed with the 
Fubini--Study metric. The vector fields $X_M$ are real parts of holomorphic sections of $T^{10}\bP_n$,
and have zeros. We conclude that the action has a moment map. By contrast, if $M$ is a torus with
translation invariant K\"ahler form, the only holomorphic section of $T^{10}M$ that has a zero is the
zero section. Indeed, it is easily seen directly that only trivial actions will have a moment map ($\equiv 0$).

Assume the $G$-action on the connected
compact K\"ahler manifold $(M,\omega)$, by holomorphic isometries, 
admits a moment map $\mu$.

\begin{prop} %2.8
If $X\in\fg$ and $h=-2\mx$, then
\begin{equation} %2.9
i\ddb h=\cL_{JX_M}\omega.
\end{equation}
\end{prop}
\begin{proof}
By definition, $X_M=-\sg h/2$, or $dh/2=\iota(X_M)\omega=\iota(X'_M)\omega+\iota(X_M'')\omega$. 
Separating $(1,0)$ and $(0,1)$ parts, $\bar\partial h=2\iota(X_M')\omega$ follows. Hence
\begin{align*}
0=d\big(2\iota(X_M')\omega-\bar\partial h\big)&= d\,\iota(X_M)\omega-\sqrt{-1}d\,\iota(JX_M)\omega-\ddb h\\
&=\cL_{X_M}\omega-\sqrt{-1}\cL_{JX_M}\omega-\ddb h.
\end{align*}
Since $\cL_{X_M}\omega=0$, (2.10) follows.
\end{proof}  
{\bf Complexifications.}
Assume $G$ is compact and connected. Then the action (2.8) extends to a holomorphic action 
\[
G^\bC\times M\ni(g,x)\mapsto gx\in M
\]
of the complexification of $G$, see e.g. \cite[III.8]{BD85} for the notion of complexification and \cite[III., Theorem 1.1]{K72},
that explains why $G^\bC$ acts. The infinitesimal generators of this action, 
real vector fields, are still 
denoted $X_M$, but this time for $X\in \fg^\bC=\fg\oplus i\fg$, the Lie algebra of $G^\bC$. Thus 
$(iX)_M=JX_M$. 
We will need the following structure theorem:
\begin{thm} %2.9
$G^\bC$ is a Stein manifold and the map
\[
G\times\fg\ni(g,X)\mapsto g\exp iX\in G^\bC\]
is a diffeomorphism.
\end{thm}
\begin{proof}
We can assume that $G$ is a subgroup of a unitary group ${\rm U}(N)$. In the construction of
$G^\bC$ in \cite{BD85} we can take $r$ the restriction of the standard representation ${\rm U}(N)\to\GL(N,\bC)$,
and the discussion after (8.2) there then identifies $G^\bC$ with an affine variety, the zero set 
$V(I)\subset\GL(N,\bC)\subset\bC^{N\times N}$ of
a certain ideal $I$ of holomorphic polynomials. In particular, $G^\bC$ is Stein. Since
$G$ is maximally real in $G^\bC$ and invariant under the operation of passing to the
adjoint of a matrix, $G^\bC=V(I)$ is also invariant.  
Hence the second statement of the theorem is a special case of \cite[Proposition 7.14]{K02}.
\end{proof}
\begin{lem} %2.10
If the $G$-action on $(M,\omega)$ admits a moment map $\mu$, then $X_M$ has a zero for all $X\in\fg^\bC$.
\end{lem}
\begin{proof} If $X=Y+iZ$ with $Y,Z\in\fg$, then $X_M'=\sg'\big(\langle\mu, Y\rangle+i\langle\mu, Z\rangle\big)$.
By \cite[Theorem 1]{LS94}, this implies that $X'_M$ vanishes somewhere, and so does $X_M$.  
\end{proof}

\begin{prop} %2.11
Assume that the $\fg$-action on $M$ is effective: $X\in\fg$ and $X_M=0$ imply $X=0$. Assume also that
the $G$-action on  $(M,\omega)$ has a moment map. If $\gamma\in G^\bC$ satisfies $\gamma^*\omega=\omega$, then
$\gamma\in G$.
\end{prop}
\begin{proof} Write $\gamma=g\exp iX$ with $g\in G$, $X\in\fg$ (Theorem 2.9). Thus $(\exp iX)^*\omega=\omega$.
By Proposition 2.10 $X$ vanishes at some $y\in M$, hence $y$ is fixed by all $g_s=\exp sX$, $s\in\bC$. Choose 
local coordinates $x_1,\dots,x_n$ centered at $y$ so that
$\omega=i\sum_j dx_j\wedge d\bar x_j$ at $y$. Passing to the Jacobian matrix defines a homomorphism
\[
\bR\ni t\mapsto \partial g_t/\partial(x_1,\dots,x_n)|_0\in \text{U}(n).
\]
The induced homomorphism of Lie algebras associates with $X$ an element $\tilde X\in\fu(n)$. We can 
arrange the coordinates so that $\tilde X$ is diagonal, diag$(i\lambda_1,\dots,i\lambda_n)$ with $\lambda_j\in\bR$. This implies 
$\partial g_s/\partial x|_0=\text{diag}(e^{i\lambda_1s},\dots, e^{i\lambda_ns})$, first for $s\in\bR$, and then for $s\in\bC$ by
analytic contiuation. Therefore at $y$
\[
\omega=(\exp iX)^*\omega=i\sum e^{-2\lambda_j}dx_j\wedge d\bar x_j,
\]
whence each $\lambda_j=0$. But then $\tilde X, X_M$, and so $X$ vanish, i.e., $\gamma=g\exp iX=g\in G$.
\end{proof}

\section{Existence of extremals} %3

From now on $(M,\omega)$ is an $n$ dimensional connected compact K\"ahler manifold, on which a connected compact
Lie group $G$ acts by holomorphic isometries. The action admits a moment map $\mu:M\to\fg^*$, and extends 
to a holomorphic action of the complexification $G^\bC$ on $M$. Without loss of generality
we assume that the induced action of
$\fg$ is effective, if $X_M=0$ for some $X\in\fg$, then $X=0$. The main result of this section is the following.
\begin{prop} %3.1
Given a bounded $v_0:M\to\bR$, among admissible $v\in\cH$, $v\le v_0$, there is at least 
one which maximizes
Monge--Amp\`ere energy $E(v)$; and among admissible $v\in\cH$, $v\ge v_0$, there is at least one 
that minimizes $E(v)$.
\end{prop}

We need two auxiliary results.
\begin{prop} %3.2
Suppose $v\in\cH$, $X\in\fg$, and $\omega_v=(\exp iX)^*\omega$. Let $g_t=\exp itX$ and
$h=-2\mx$. Then 
\[
\varphi(t)=\int_0^tg_\tau^*h\,d\tau\qquad\text{and}\qquad \psi(t)=\varphi(t)+E(v)t
\]
define geodesics $\varphi:\bR\to\cK$, $\psi:\bR\to\cH$. For each $t$, $\varphi(t),\psi(t)$ are admissible and $\psi(1)=v$.
\end{prop}
\begin{proof}
All claims but the last follow by combining Propositions 2.5, 2.8, and (2.6). As to the last, (2.6) gives
$\omega_{\psi(1)}=(\exp iX)^*\omega=\omega_v$, whence $\psi(1)=v+\text{const}$. The constant has to
be $0$ because $E\big(\psi(1)\big)=E(v)$.
\end{proof}
\begin{prop} %3.3
If $a,b\in\bR$, 
\begin{align}
K=&\{g\in G^\bC: g^*\omega=\omega_v\text{ with }v\in\cH\text{ such that } a\le E(v),\,\max_M v\le b\}, \\%3.1
L=&\{g\in G^\bC: g^*\omega=\omega_v\text{ with }v\in\cH\text{ such that } a\le\min_M v,\, E(v) \le b\}
\end{align}
are compact subsets of $G^\bC$.
\end{prop}
\begin{proof}
We start by defining a norm on $\fg$.  Let $\Delta=2\,\tr_\omega i\ddb$ denote the Laplacian on 
$C^\infty(M,\omega)$. If $X\in\fg$ and $h_X=-2\mx$, by Proposition 2.8 
$i\ddb h_X=\cL_{JX_M}\omega$.  Taking trace with respect to $\omega$, we obtain 
$\Delta h_X=2\,\tr_\omega \cL_{JX_M}\omega$, hence
with the Green operator $\Gamma:C^\infty(M,\omega)\to C^\infty(M,\omega)$
\begin{equation} %3.3
h_X=\Gamma(2\,\tr_\omega\cL_{JX_M}\omega).
\end{equation}
Let
\[
||X||=\int_M|h_X|\omega^n.
\]
This is indeed a norm, for if $||X||=0$, then $h_X$ and $X_M=-\sg h_X/2$ vanish, whence $X=0$.

Now to the compactness of $K,L$. Write an arbitrary $g\in G^\bC$ as $g=\gamma\exp iX$, $\gamma\in G$,
$X\in\fg$ (Theorem 2.9). Thus $g^*\omega=(\exp iX)^*\omega$. If  $\omega_v=g^*\omega$ then $X$,
i.e. $h_X$, can be estimated in terms of  $\min_M v$, $\max_M v$, and $E(v)$ as follows. Let $g_t=\exp itX$, 
and with
$h_X$ defined in (3.3), $\psi(t)=\int_0^tg_\tau^*h_X\,d\tau+E(v)t$. By  Proposition 3.2 this is a geodesic.
In particular, $\psi(\cdot)(x)$ is convex for $x\in M$, whence
\begin{equation*} 
h_X+E(v)=\dot\psi(0)\le\psi(1)-\psi(0)\le\dot\psi(1)=g_1^*h_X+E(v).
\end{equation*}
Since $\psi(0)=0$ and $\psi(1)=v$ by Proposition 3.2, we obtain
\begin{equation} %3.4
h_X\le v-E(v)\le g_1^*h_X.
\end{equation}

Suppose $g\in K$. Then (3.4) implies  $h_X\le b-a$; with $h^+=\max(h_X,0)$
\[
||X||=\int_M|h_X|\,\omega^n=\int_M(2h^+ -h_X)\,\omega^n=\int_M 2h^+\omega^n\le2(b-a)\int_M\omega^n.
\]
Similarly, if $g\in L$, and $h^-=\max(-h_X,0)$, then (3.4) gives $\min_M h_X=\min_M g_1^*h_X\ge a-b$, and
\[
||X||=\int_M|h_X|\,\omega^n=\int_M(2h^-+h_X)\,\omega^n=\int_M 2h^-\omega^n\le 2(b-a)\int_M\omega^n.
 \]
 These estimates show that the closed sets $K,L\subset G^\bC$ are contained in the compactum
 \[
 \Big\{\gamma\exp iX:\gamma\in G,\, X\in\fg,\, ||X||\le 2(b-a)\int_M\omega^n\Big\},
 \]
 therefore themselves are compact.
\end{proof}
\begin{proof}[Proof of Proposition 3.1]
Since $G^\bC$ leaves the de Rham class of $\omega$ invariant, for each $g\in G^\bC$ there is a unique
$u=U(g)$ that solves
\begin{equation} %3.5
g^*\omega-\omega=i\ddb u,\qquad \int_M u\,\omega^n=0.
\end{equation}
As in the previous proof, the map $U:G^\bC\to C^\infty(M)$ can be expressed through the Green operator 
$\Gamma$ as
\[
U(g)=\Gamma\big(2\,\tr_\omega(g^*\omega-\omega)\big),
\]
which shows that $U$ is continuous. Let $m(g)=\min_M\big(v_0-U(g)\big)$. Suppose
$g^*\omega=\omega_v$. Then $v\le v_0$ if and only if $v=U(g)+m(g)+c$ with a constant $c\le 0$. As,
in general, $E(w+c)=E(w)+c$, maximizing $E(v)$ among admissible $v\le v_0$ reduces to maximizing
\begin{equation} %3.6
E\big(U(g)+m(g)\big),
\end{equation}
a continuous function, over $G^\bC$. Let $a=\inf_M v_0$ and $b=\sup_M v_0$. Since the constant
function $a\le v_0$ is admissible, maximizing (3.6) over $G^\bC$ amounts to the same as maximizing over
$K$, see (3.1); and of course, the maximum over this compact set 
is attained. The existence of the minimizer
in the second statement follows similarly.
\end{proof}

\section{Characterization of the extremals} %4

We continue with the setup of section 3. To investigate extremals of problems (1.2), (1.3),
it will be convenient to be able to change the K\"ahler form $\omega$ to $\omega'=\omega_w$ with 
$w\in\cH$. We start the section by discussing how notions involved in Theorems 1.2, 1.3 transform as 
we pass from $\omega$ to $\omega'$.

Let $\cH'=\cH(\omega')=\{u\in C^\infty(M): \omega'_u=\omega'+i\ddb u>0\}$. Consider the map
\begin{equation} %4.1
\cH\ni v\mapsto v'=v-w\in\cH'.
\end{equation}
One checks that $\omega'_{v'}=\omega_v$, and that (4.1) is an isometry in Mabuchi's metric. Denote
Monge--Amp\`ere energy on $\cH'$ by $E':\cH'\to\bR$. Given $v\in\cH$, let $\psi:[0,1]\to\cH$ be a smooth
path connecting $w$ with $v$, and let $\psi'(t)=\psi(t)-w$. Continuing to denote $t$-derivatives by a dot,
\begin{equation}\begin{aligned} %4.2
E'(v')&=\int_0^1dE'\big(\dot\psi'(t)\big)\,dt=\dfrac1\vo\int_0^1\int_M\dot\psi'(t)\big(\omega'_{\psi'(t)}\big)^n\,dt\\
&=\dfrac1\vo\int_0^1\int_M\dot\psi(t)\omega_{\psi(t)}^n\,dt=\int_0^1dE\big(\dot\psi(t)\big)\,dt=E(v)-E(w).
\end{aligned}\end{equation}
Hence finding extremals of $E$ and $E'$ is one and the same thing.

Next suppose $w\in\cH$  is admissible, $\omega'=\omega_w=\gamma^*\omega$ with some 
$\gamma\in G^\bC$. 
The action
\begin{equation} %4.3
G^\bC\times M\ni(g,x)\mapsto\gamma^{-1}g\gamma\, x\in M,
\end{equation}
and its restriction to $G$, will be referred to 
as the modified action. If $v\in\cH$ is admissible, i.e., $\omega_v=g^*\omega$ with $g\in G^\bC$, then
$\omega'_{v'}=\omega_v=(\gamma^{-1}g)^*\omega'$ shows that $v'\in\cH'$ 
is also admissible; and 
the converse holds as well.

The map $\gamma:(M,\omega')\to(M,\omega)$ is an isomorphism of K\"ahlerian $G^\bC$-spaces, if
$(M,\omega')$ is endowed with the modified action (4.3)
(and $(M,\omega)$ with the original action).
In particular, the modified action of $G$ is by isometries of $(M,\omega')$. It follows that the moment map
of the modified action is $\mu'=\mu\circ\gamma$. If $\gamma\in G$, then $\mu'=\text{Ad}_{g^{-1}}^*\,\mu$ by equivariance.
To summarize:

\begin{prop} %4.1
Suppose $w\in\cH$ is admissible, $\omega_w=\gamma^*\omega$, $\gamma\in G^\bC$. The map
$\cH\ni v\mapsto v-w\in\cH'$ is an isometry that sends Monge--Amp\`ere energy $E$ on $\cH$ to Monge--Amp\`ere
energy $E'$ on $\cH'$, plus a constant. Let
$v_0:M\to\bR$ and $v'_0=v_0-w$. An admissible $u\in \cH$ maximizes Monge--Amp\`ere energy $E(v)$
among admissible $v\in\cH$,  $v\le v_0$, if and only if $u'=u-w$ maximizes $E'(v')$ among
admissible $v'\in\cH'$, $v'\le v_0'$. For minimizers of $E$ over $\{v\in\cH\text{ admissible }:v\ge v_0\}$
a corresponding statement holds.

Furthermore, the moment maps $\mu:(M,\omega)\to\fg^*$, $\mu':(M,\omega')\to\fg^*$ for the original and modified
$G$-actions (see (4.3)) are related by $\mu'=\mu\circ\gamma$. If $\gamma^*\omega=\omega$, then
$\gamma\in G$ (by Proposition 2.11), and
$\mu'=\text{Ad}_{\gamma^{-1}}^*\,\mu$. 
\end{prop}

Now suppose $v_0\in C(M)$. If $u\in\cH$, let $C_u=\{x\in M: u(x)=v_0(x)\}$.
\begin{prop} %4.2 
Suppose $u\in\cH$ is admissible, $\omega_u=\gamma^*\omega$ with $\gamma\in G^\bC$. If $u=v$
maximizes $E(v)$ among admissible $v\le v_0$, or if $u=v$ minimizes $E(v)$ among admissible
$v\ge v_0$, then the convex hull of $\mu(\gamma C_u)\subset\fg^*$ contains $0$.
\end{prop}
\begin{proof}
Consider first the case when $u=0$. When $u$ minimizes, $v_0\le 0$, when $u$ maximizes, $v_0\ge 0$.
Since $\omega=\omega_u=\gamma^*\omega$, by Proposition 4.1
$\mu\circ\gamma=\text{Ad}_{\gamma^{-1}}^*\,\mu$. Hence the convex hull of 
$\mu(\gamma C_u)\subset\fg^*$ contains $0$
if and only if the convex hull of $\mu(C_u)\subset\fg^*$ contains $0$.
Suppose this latter does not contain $0$. Then there is a linear form
$l:\fg^*\to\bR$ such that $l\big(\mu(x)\big)<0$ for all $x\in C_0$. Such a form $l$ is evaluation at some
$X\in\fg$. Thus 
\begin{equation} %4.4
h=-2\mx>0 \qquad \text{on } C_0,
\end{equation}
and in fact $U=\{x\in M:h(x)>\var\}$ is a neighborhood of $C_0$ for some $\var>0$. Let $g_t=\exp itX$ and
$\varphi(t)= \int_0^tg_\tau^*h\,d\tau$. By Proposition 3.2 
$E\big(\varphi(t)\big)= 0$ and  $\varphi(t)$ is admissible for all $t$. Fix a compact neighborhood $K\subset U$ of $C_0$. 
Since $h|U>\var$, for small $t\in\bR$ and $y\in K$
\[
\varphi(t)(y)\ge\var t \text{ if } t>0,\qquad \varphi(t)(y)\le\var t\text{ if }t<0.
\]

If $u$ minimizes (maximizes), then 
$\varphi(0)(x)=0>v_0(x)$, respectively $\varphi(0)(x)=0<v_0(x)$, when
$x\in M\setminus \text{int}\,K$. It follows that for small $t$ and on all of $M$
\begin{align*}
&\varphi(t)>v_0,\quad\text{ if } t>0 \text{ and }u=0\text{ minimizes},\\
&\varphi(t)<v_0,\quad\text{ if } t<0 \text{ and }u=0\text{ maximizes.}
\end{align*}
The inequalities continue to hold if a small constant $c$ is added to $\varphi(t)$. But this contradicts 
the extremality of $u=0$, since $\varphi(t)+c$ is also admissible, and $E\big(\varphi(t)+c\big)=c=E(u)+c$.

Second, consider a general extremal $u$. Let $\omega'=\omega_u$,
$v_0'=v_0-u$, and $E':\cH'\to\bR$ the Monge--Amp\`ere energy as in Proposition 4.1 (but with $w=u$). Since
the contact sets $\{u=v_0\}$ and $\{0=v_0'\}$ agree, Proposition 4.1 and the special case we just proved 
imply the necessity of the condition in Proposition 4.2 in general. 
\end{proof}

Next we prove sufficiency.
\begin{prop} %4.3
Suppose $u\in\cH$ is admissible, $\omega_u=\gamma^*\omega$ with $\gamma\in G^\bC$; and the
convex hull of $\mu(\gamma C_u)\subset\fg^*$
contains $0$. If $u\le v_0$, then $u=v$ maximizes $E(v)$ among all admissible $v\le v_0$. If $u\ge v_0$
and $v_0\in\po$, then $u=v$ minimizes $E(v)$ among all admissible $v\ge v_0$.
\end{prop}
The convex hull condition is equivalent to the existence of a Borel probability measure $m$ (even one with finite
support) on the contact set $C_u$ such that
\begin{equation} %4.5
\int_{C_u}(\mu\circ\gamma)\,dm=0.
\end{equation}
\begin{proof}
Again it suffices to prove under the assumption $u=0$ and $\gamma=e$; the general case will follow in
light of Proposition 4.1. Suppose $v\in\cH$ is another admissible potential, $\omega_v=g^*\omega$.
By Theorem 2.9 $g$ can be taken of form $g=\exp iX$, $X\in\fg$. 
With $h=-2\mx$, by Proposition 3.2
\[
\psi(t)=\int_0^t(\exp i\tau X)^*h\,d\tau+E(v)t, \qquad 0\le t\le 1,
\]
defines a geodesic connecting $0$ with $v$. By convexity (Proposition 2.2), if $0,v\le v_0$ then 
$\psi(t)\le v_0$; if $0,v\ge v_0$ then $\psi(t)\ge v_0$, for $0\le t\le1$. Also $\psi(0)(x)=0$ when $x\in C_0$.

In the first case of the proposition it follows that $0\ge\dot\psi(0)(x)=h(x)+E(v)$ for $x\in C_0$. But (4.5) implies 
$\int_{C_0}h\,dm=0$, and so 
\[
E(u)=0\ge\int_{C_0}\big(h+E(v)\big)\,dm=E(v),
\]
as claimed. In the second case we have $0\le\dot\psi(0)(x)=h(x)+E(v)$ for $x\in C_0$. 
Therefore 
\[
E(u)=0\le\int_{C_0}\big(h+E(v)\big)\,dm=E(v),
\]
which completes the proof.
\end{proof}

\section{Uniqueness of extremals} %5

Let $(M,\omega)$, the $G$- and $G^\bC$-actions, and the moment map $\mu:M\to\fg^*$ be as in sections 3, 4. 
With $X\in\fg$ put
\[
N_X=\{x\in M:\mx\text{ vanishes to second order at }x\},
\]
and introduce the assumption
\begin{equation} %5.1
\text{If }X\in\fg\setminus\{0\},\text{ the convex hull of }\{\mu(x):x\in N_X\}\subset\fg^*\text{ avoids }0.
\end{equation}
This latter is the same as requiring that under some linear map $\fg^*\to\bR$ the image of $\mu(N_X)$ be
on one side of $0\in\bR$. Thus (5.1) is equivalent to
\begin{equation} %5.2
\text{If } X\in\fg\setminus\{0\},\text{ there is a }Y\in\fg\text{ such that }\langle\mu,Y\rangle>0\text{ on }N_X.
\end{equation}

Here are two examples. We use dagger $^\dag$ to indicate adjoint of a matrix.
The standard action of $G=\SU(n+1)$ on $\bC^{n+1}$ descends to an action
on projective space $M=\bP_n=\big(\bC^{n+1}\setminus\{0\}\big)/\sim$, preserving the Fubini--Study form 
$\omega$. The Lie algebra $\fg$ consists of traceless skew adjoint matrices $X$, $X^\dag=-X$. By
Kirwan \cite[Lemma 2.5]{K84}, 
for example, the pull back $\tilde\mu$ of the moment map $\mu:\bP_n\to\fg^*$ by the projection
$\bC^{n+1}\setminus\{0\}\to\bP_n$ is
\begin{equation} %5.3
\langle\tilde\mu(z),X\rangle=-i\dfrac{z^\dag Xz}{z^\dag z},\qquad z\in \bC^{n+1}\setminus\{0\}
\text{ a column vector},\, X\in\fg,
\end{equation}
up to a positive constant factor that is determined by the normalization of the Fubini--Study metric.

Assumption (5.2) is satisfied. It suffices to verify this when $X$ is diagonal with diagonal entries 
$ia_0,\dots,ia_n\in i\bR$,
not all $0$. The simultaneous vanishing of  $\langle\tilde\mu(z),X\rangle=\sum a_j|z_j|^2/\sum|z_j|^2$ and 
$d\langle\tilde\mu(z),X\rangle(JX_M)=$
\[
\partial_t|_{t=0}\langle\tilde\mu\big((\exp itX)z\big),X\rangle
=2\,\frac{\sum a_j^2|z_j|^2\sum|z_j|^2-(\sum a_j|z_j|^2)^2}{(\sum|z_j|^2)^2}
\]
implies $a_jz_j=0$ for all $j$. Say, $a_0\neq 0$. Then $x=(x_0:\ldots:x_n)\in N_X$ implies $x_0=0$, and 
so the diagonal matrix $Y\in\fg$ with entries $-in,i,\dots,i$ satisfies $\langle\mu,Y\rangle>0$ on $N_X$.

As a second example, consider the unit sphere $S^2\subset\bR^3$, with the inherited conformal, hence
complex, structure and rotation invariant K\"ahler form (=Fubini--Study form, if $S^2$ is identified
with $\bP_1$). Let $M=(S^2)^n$ and $\omega$ the product K\"ahler form on $M$. The group
$G=\text{SO}(3)$ acts on $M$ diagonally by rotations. If $\fg^*$ is suitably
identified with $\bR^3$, the moment map is
\[
\mu(x_1,\dots,x_n)=x_1+\dots+x_n\in\bR^3\approx\fg^*,\qquad (x_j)\in M,
\]
see \cite[2.8]{K84}, and $\mx$ is the composition of $\mu$ with a linear $\pi:\bR^3\to\bR$. Its critical
points occur where each $x_j\in S^2$ is an extremum point of $\pi|S^2$. This means each $x_j$ must agree
with one of a fixed pair of antipodal points in $S^2$. If $n$ is odd then $\mx\neq 0$ on the critical set,
$N_X=\emptyset$, and (5.1) is satisfied. If $n$ is even then $N_X$ is nonempty, and along with any
$x$ it contains $-x$ as well. Since $\mu(x)+\mu(-x)=0$, (5.1) fails.

Now to the uniqueness in problems (1.2), (1.3). Let $v_0:M\to\bR$ be continuous.
\begin{prop} %5.1
Assume (5.1) (or equivalently (5.2)). There is a unique  $v$ that maximizes $E(v)$ among
admissible $v\le v_0$. If $v_0\in C^2(M)$ and 
$\omega+i\ddb v_0>0$, then there is a unique $v$ that minimizes
$E(v)$ among admissible $v\ge v_0$.
\end{prop}

Later we will need that in the minimization problem a more liberal if slightly technical condition also
implies uniqueness:

\begin{prop} %5.2
Assume (5.1). Consider a $v_0\in\po$ and an admissible $u\ge v_0$, $\omega_u=\gamma^*\omega$ with some
$\gamma\in G^\bC$. If there is a Borel probability measure $m$ on the contact set $C_u=\{u=v_0\}$ such that 
$\int_{C_u}(\mu\circ\gamma)\,dm=0$, $v_0$ is $C^2$ in a neighborhood of $\supp m$, and $\omega+i\ddb v_0>0$
there, then $u$ is the unique minimizer of $E(v)$ among admissible $v\ge v_0$.
\end{prop}
\begin{proof}[Proof of Propositions 5.1, 5.2] 
In Proposition 5.1 the existence of extremals is the content of Proposition 3.1. As to uniqueness, the statement of
Proposition 5.1 about minimizers follows from Proposition 5.2. Indeed, Proposition 4.2 implies that if in the scenario of
Proposition 5.1 an admissible $u\ge v_0$ minimizes energy, then there is a Borel probability measure $m$ on $C_u$
such that $\int_{C_u}(\mu\circ\gamma)\,dm=0$. Since $v_0\in C^2(M)$ and $\omega+i\ddb v_0>0$ everywhere, 
Proposition 5.2 applies and gives uniqueness.

This leaves us with proving uniqueness of maximizers in Proposition 5.1, and Proposition 5.2. We will refer to the former 
as Case A, to the latter as Case B. In what follows, let $u,\bar u\le v_0$ be two maximizers in Case A, 
$E(u)=E(\bar u)$; or $u$ as in 
Case B and $\bar u\ge v_0$ admissible, $E( u)\ge E(\bar u)$. In both cases we need to show $u=\bar u$. 

In light of Proposition 4.1 we can assume $ u=0$ and, in case B, $\gamma=e$. 
Let $\omega_{\bar u}=g^*\omega$, $g\in G^\bC$. By Theorem 2.9 we can also assume
$g=\exp iX$, $X\in\fg$. With $g_t=\exp itX$ and $h=-2\mx$, 
\begin{equation} %5.4
\psi(t)=\int_0^{t}g_\tau^*h\,d\tau+tE(\bar u),\qquad 0\le t\le 1,
\end{equation}
defines a geodesic in $\cH$, 
% $E\big(\varphi(t)\big)\equiv 0$, 
$\psi(0)=0$ and $\psi(1)=\bar u$ (Proposition 3.2).
By convexity, Proposition 2.2, $\psi(t)\le v_0$ in  Case A and $\psi(t)\ge v_0$ in Case B.
In Case A $u=0$ is known to be a maximizer. Therefore by Proposition 4.2 we can find a Borel probability measure 
$m$ on $C_0$, that we fix, such that $\int_{C_0}\mu\,dm=0$. We claim that in both cases $\psi(\cdot)(y)$ is constant for
$m$--almost every $y\in C_0$.

Let us start with Case A. With any $y\in C_0$, $\psi(\cdot)(y)$ attains its maximum at $0$, whence 
$0\ge\dot\psi(0)(y)=h(y)$. Since $\int_{C_0}h\,dm=-2\int_{C_0}\mx\,dm=0$, for $m$--almost every $y$ we have
$\dot\psi(0)(y)=0$. But $\psi(\cdot)(y)$ is convex, hence it must be constant, as claimed.

In case B, if $y\in C_0$, $\psi(\cdot)(y)$ has a minimum at $0$, so $0\le\dot\psi(0)(y)=h(y)+E(\bar u)\le h(y)$. Again
$\int_{C_0} h\,dm=0$ implies $\dot\psi(0)(y)=0$ for $m$--almost every $y$.
Pick a smooth function $w$ on a neighborhood of $y\in M$ such that
$i\ddb w=\omega$. The function $\Psi(s,x)=\psi(\re s)(x)$, $s\in \bar S_{01}$, $x\in M$, solves
\[
\text{rk}\, (\pi^*\omega+i\ddb\Psi)\equiv n\qquad\text{on } \bar S_{01}\times M,
\]
see section 2. Hence $\text{rk}\,\ddb(\pi^* w+\Psi)\equiv n$ in a neighborhood of $(0,y)\in \bar S_{01}\times M$. 
The leaf of the associated Monge--Amp\`ere foliation \cite{BK77} through $(0,y)$ contains the graph
of a smooth function 
$f:U\to M$, where $U\subset\bar S_{01}$ is a connected (relative) neighborhood of $0\in\bar S_{01}$, and $f$ is 
holomorphic on $U\cap S_{01}$. Thus $f(0)=y$ and 
$(\pi^*w+\Psi)\big(s,f(s)\big)$ is a smooth function of  $s\in U$, harmonic when $\re s>0$. As
$w+v_0$ is plurisubharmonic, $(w+v_0)\circ f$ is subharmonic in $U\cap S_{01}$. Also
\begin{equation} %5.5
(w+v_0)\big(f(s)\big)\le w\big(f(s)\big)+\Psi\big(s,f(s)\big)=(\pi^*w+\Psi)\big(s,f(s)\big),
\end{equation}
with equality when $s=0$. But in fact, if $t=\re s$, the $\partial_t$--derivative of the two sides of (5.5) also agree 
when $s=0$, because $v_0\le\Psi(0,\cdot)$ with equality at $y$, hence $dv_0=d\Psi(0,\cdot)$ at $y$; and
$\partial_t|_{t=0}\Psi(t)(y)=\dot\psi(0)(y)=0$ for $m$--almost every $y$. To sum up, for $m$--almost every $y$
the left hand side of (5.5) minus the right hand side, a subharmonic function of $s$, attains its maximum $=0$
at $0$, and there its normal derivative is also $0$. The Hopf--Oleinik lemma, which in the case needed apparently
goes back to  Zaremba \cite{H52, O52, Z10}, then implies that the difference function is $\equiv 0$, and
equality holds in (5.5) for all $s\in U$.
In particular, $(w+v_0)\circ f$ is harmonic; since $i\ddb(w+v_0)>0$ near $y$, it follows that $f$ is constant, $f\equiv y$.
But then (5.5), with equality instead of inequality, gives $v_0(y)=\psi(t)(y)$ for $t\in U\cap\bR$, and
by analytic continuation, for all $t\in[0,1]$.

With the claim proved, we compute $0=\dot\psi(t)(y)=h(g_{t}y) $,
\begin{equation*}
0=\ddot\psi(t)(y)=\partial_\tau|_{\tau=0}h\big(g_\tau(g_ty)\big)=\iota\big(JX_M(g_ty)\big)dh=
2\iota\big(JX_M(g_ty)\big)\iota\big(X_M(g_ty)\big)\omega,
\end{equation*}
as $dh=2\iota(X_M)\omega$. Since $\omega>0$, this implies $X_M(g_{t}y)=0$.
In particular, $\mx=-h/2$ vanishes to second order at $g_{0}y=y$, i.e., $m(C_0\setminus N_X)=0$.
On the one hand, this means that $0=\int_{C_0}\mu\,dm=\int_{C_0\cap N_X}\mu\,dm$, and so $0$ is in the
convex hull of $N_X$; on the other, by (5.1) this cannot happen unless $X=0$. But then $h=0$ and $\bar u=\psi(1)=0$
follow.
\end{proof}

Putting together Propositions 3.1, 4.2, 4.3, and 5.1 we obtain Theorems 1.2, 1.3.

If in Theorems 1.2, 1.3, or in Proposition 5.1 condition (5.1) is dropped, 
the extremals may no longer be unique.
For example:
\begin{prop} %5.3
Suppose that for some nonzero $X\in\fg$ the convex hull of $\{\mu(x):x\in N_X\}\subset\fg^*$ contains $0$.
Then there is a $v_0\in\cH$ such that the minimum of $E(v)$ among admissible $v\ge v_0$ is attained
at more than one $v$.
\end{prop}
\begin{proof}
Write $\dist$ for distance on $M$ measured in, say, the K\"ahler metric of $\omega$.
Let $g_t=\exp itX$, $h=-2\mx$, and define a geodesic $\varphi:\bR\to\cH$ by 
$\varphi(t)=\int_0^tg_\tau^*h\,d\tau$ (Proposition 3.2). Thus 
$g_t|N_X=\id_{N_X}$ and $\varphi(t)|N_X=0$. First observe that 
for $t\in[-1,1]$ the functions $g_t^*h$ vanish to
second order on $N_X$, and have uniformly bounded Hessians on $M$.
Hence 
\[
|\dot\varphi(t)(z)|=|(g_t^*h)(z)|=O\big(\dist(z,N_X)^2\big),\qquad t\in[-1,1], \quad z\in M.
\]
Since
$\varphi(0)=0$, integration gives 
\begin{equation} %5.6
|\varphi(t)(z)|=O\big(|t|\dist(z, N_X)^2\big),\qquad t\in[-1,1],\quad z\in M.
\end{equation}

According to Atiyah, Guillemin, and Sternberg \cite[Lemma 2.2]{A82}, \cite[Theorem 5.3]{GS84}, $N_X$,
the union of certain components of the critical set of $h$, is a smooth submanifold of $M$. (The weaker
result \cite[Theorem 5.1]{K72} would also do.) Choose  a smooth
$v_0:M\to(-\infty,0]$ whose zero set is $N_X$, and whose Hessian in directions normal to $N_X$ is
negative definite. If needed, we can replace $v_0$  by $\var v_0$ to arrange that $v_0\in\cH$. Since $v_0$
is negative away from $N_X$, the Hessian condition implies
$v_0(z)\le c\,\dist(z,N_X)^2$ for $z\in M$, with some $c<0$. Hence (5.6) yields $\varphi(t)\ge v_0$
for small $t$. Each $\varphi(t)$ is admissible by Proposition 3.2.

But $0=\varphi(0)$ minimizes energy $E(v)$ among admissible $v\ge v_0$, because by assumption
the convex hull of $\{\mu(x):x\in N_X\}=\{\mu(x):0=v_0(x)\}$ contains $0$, cf. Proposition 4.3. As
$E\big(\varphi(t)\big)\equiv 0$ by Proposition 3.2, for small $t$ the potentials $\varphi(t)$ also minimize.
\end{proof}

%Of course, even if (5.1) fails, for some $v_0$ the extremals in problems (1.2), (1.3) may be unique. One can
%check that the proof of Proposition 5.1, say, for minimizing $E$, works under the assumption, weaker than (5.1), that for each
%$g\in G^\bC$ and nonzero $X\in\fg$ the convex hull of 
%\[
%\{\mu(x): g^*v_0\text{ has a maximum at }x\in N_X\}
%\]
%avoids $0$.

\section{Stability of extremals, extrema} %6

We continue in the setup of section 5. In view of Proposition 5.1, with any $v_0\in C(M)$ one can associate 
the unique admissible $u_-=U_-(v_0)$ that
maximizes $E(v)$ among admissible $v\le v_0$. Similarly, if we denote by $C_\text{unique}\subset C(M)$ the
set of $v_0$  such that $E$ has a unique minimizer among admissible $v\ge v_0$, with each $v_0\in C_\text{unique}$
we can associate this unique minimizer $u=U(v_0)$. For example, if the assumptions of Proposition 5.1 are
met, $\cH\subset C_\text{unique}$.
First we record a stability property of $u_-,u$; then of the minimum value in (1.3). 
\begin{prop} %6.1
The maps 
\[
U_-:C(M)\to\cH\qquad\text{and}\qquad U:C_\text{unique}\to\cH
\]
are continuous, when their domains are endowed with the topology of uniform convergence and $\cH\subset C^\infty(M)$
with its natural Fr\'echet topology.
\end{prop}
\begin{proof}
Since the proofs for $U_-$ and $U$ are almost identical, we will write out the proof only for $U$. Suppose 
$v_0,v_1\in C_\text{unique}$, and
$\max_M|v_0-v_1|=\delta$. Thus $U(v_1)\ge v_1\ge v_0-\delta$, whence the admissible potential  
$U(v_1)+\delta$ is $\ge v_0$. Therefore 
\[
E\big(U(v_1)\big)+\delta=E\big(U(v_1)+\delta\big)\ge E\big(U(v_0)\big).
\]
Exchanging the roles of $v_0,v_1$, we obtain $\big|E\big(U(v_0)\big)-E\big(U(v_1)\big)\big|\le\delta$, and so 
$E\circ U$ is continuous.

Now take a sequence of $v_j\in C_\text{unique}$, $j=1,2,\dots$, converging uniformly to $v_0\in C_\text{unique}$. Thus 
$\omega_{U(v_j)}=g_j^*\omega$ with some $g_j\in G^\bC$, $j=0,1,\dots$. To prove $\lim_jU(v_j)= U(v_0)$ 
we first assume $g_j\to g_0$. By Theorem 2.9 $g_j=\gamma_j\exp iX_j$ with $\gamma_j\in G$, $X_j\in\fg$, and
$X_j\to X_0$. Since $(\exp iX_j)^*\omega=\omega_{U(v_j)}$, by Proposition 3.2
\[
U(v_j)=-2\int_0^1(\exp i\tau X_j)^*\langle \mu, X_j\rangle\,d\tau+E\big(U(v_j)\big).
\]
This indeed shows that $U(v_j)\to U(v_0)$.

Without assuming $g_j\to g_0$ we can argue like this. Suppose $U(v_j)\not\to U(v_0)$. Thus $U(v_0)$ 
has a neighborhood
$\cV\subset\cH$ that infinitely many $U(v_j)$ avoid. By compactness, Proposition 3.3, at the price of passing to a 
subsequence, we can arrange that the corresponding $g_j$ converge. By what we have already proved, this implies
that the $U(v_j)$ converge to some admissible $u\neq U(v_0)$. However, $u\ge\lim_j v_j= v_0$ and
$E(u)=\lim_j E\big(U(v_j)\big)=E\big(U(v_0)\big)$; this contradicts uniqueness of the minimizer.
\end{proof} 

In section 9 we will need a coarser stability property of the minimum value itself in (1.3). For any upper bounded
$v_0:M\to[-\infty,\infty)$ let
\[
I(v_0)=\inf\{E(v):v\in\cH \text{ is admissible, }v\ge v_0 \}\in[-\infty,\infty).
\]
\begin{prop} %6.2
If a sequence of upper semicontinuous $v_k:M\to[-\infty,\infty)$, $k\in\bN$, decreases to $v_0$, then
$I(v_k)$ decrease to $I(v_0)$.
\end{prop}
\begin{proof}
The definition implies that $I(v_k)\ge I(v_0)$, $k\in\bN$, form a decreasing sequence; call its limit $L\ge I(v_0)$.
Suppose $u\ge v_0$ is admissible. Then the upper semicontinuous functions $\max(v_k,u)$ decrease to $u$, 
perforce uniformly by Dini's theorem. Thus, if $\var>0$ there is a $k$ such that $v_k\le u+\var$, whence
$L\le I(v_k)\le E(u+\var)=E(u)+\var$. 
This being true for all $\var>0$ and for all admissible $u\ge v_0$, $L\le I(v_0)$ follows, and the proof is complete.
\end{proof}

\section{Polarized manifolds} %7

With $(M,\omega)$, $G$, and $G^\bC$ as before we now consider a holomorphic line bundle $L\to M$ with a smooth
hermitian metric $h^0$. The curvature $\Theta(h^0)$ is a closed form, and we assume $\omega=i\Theta(h^0)$. We also assume
that the action of each $g\in G^\bC$ on $M$ can be lifted to $\tilde g:L\to L$; thus $\tilde g$ is holomorphic and maps
each $L_x$ linearly onto $L_{gx}$. This section is about a reformulation of the variational problems (1.2), (1.3) in this
setting. The material, when $L$ is the hyperplane section bundle over projective space, will be needed in sections 8, 10.

Any two lifts $L\to L$ of $g\in G^\bC$ are multiples of one another by a holomorphic function on $M$, hence by a nonzero
constant. All lifts
of all $g\in G^\bC$ form a group $H^\bC$, an extension of $G^\bC$ by $\bC^\times$, the multiplicative
group of nonzero complex numbers. While $H^\bC$ is the complexification of an extension $H$ of $G$ by $S^1$, we will not need
this fact. We can let $H^\bC$ act on $M$ as well, via the projection $H^\bC\to G^\bC$.

For each $g\in H^\bC$ there is a $W(g)\in C^\infty(M)$ such that
\begin{equation} %7.1
g^*h^0=e^{-W(g)}h^0.
\end{equation}
(In this formula, as usual, $W(g)$ is viewed as a function on $L$, constant on the fibers.) Thus
\[
0<i\Theta(g^*h^0)=i\Theta(h^0)+i\ddb W(g)=\omega_{W(g)},
\]
$W(g)\in\cH$, and $W$ is a smooth map $H^\bC\to\cH$. Since $i\Theta(g^*h^0)=ig^*\Theta(h^0)=g^*\omega$, it follows 
that $W(g)$ is admissible in the sense of Definition 1.1. The converse is also true:
\begin{prop} %7.1
$v\in\cH$ is admissible if and only if it is of form $W(g)$, where $g\in H^\bC$.
\end{prop}
\begin{proof}
What remains to prove is the ``only if'' statement. Note that if $\lambda\in \bC^\times\subset H^\bC$ then
\[
e^{-W(\lambda g)}h^0=(\lambda g)^*h^0=g^*\lambda^*h^0=|\lambda|^2g^*h^0=|\lambda|^2e^{-W(g)}h^0.
\]
Hence $W(\lambda g)=W(g)-\log|\lambda|^2$. Suppose that $v$ is admissible, $\omega_v=g_0^*\omega$
with some $g_0\in G^\bC$. Choose $g_1\in H^\bC$ that projects to $g_0$. By (7.1)
\[
\omega+i\ddb v=g_1^*\omega=i\Theta(g_1^*h^0)=i\Theta(e^{-W(g_1)}h^0)=\omega+i\ddb W(g_1),
\]
whence $W(g_1)-v=c$ is constant. It follows that $W(e^{c/2}g_1)=W(g_1)-c=v$.
\end{proof}

\begin{defn} %7.2
We say that hermitian metrics $h,h^0$ on $L$ are similar, or are in the same similarity class, if $h=g^*h^0$ with some 
$g\in H^\bC$. In this case we put $\cE(h^0,h)=E\big(W(g)\big)=E\big(\log(h_0/h)\big)$.
\end{defn}
An arbitrary metric $\bar h$ on $L$ is of form $e^{-\bar v}h^0$, with $\bar v$ a function on $M$. By our discussion: 
\begin{prop} %7.3
The extrema in (1.2), (1.3) can be represented as
\begin{equation}
\begin{aligned} %7.2
\max\{E(v):v\le \bar v\text{ is admissible}\}=& \max\{E\big(W(g)\big): g\in H^\bC, W(g)\le \bar v\} \\
=&\max\{\cE(h^0,h): h\ge\bar h\text{ is similar to }h^0\}
\end{aligned}
\end{equation}
and
\begin{equation}
\begin{aligned} %7.3
\min\{E(v):v\ge \bar v\text{ is admissible}\}=& \min\{E\big(W(g)\big): g\in H^\bC, W(g)\ge \bar v\} \\
=&\min\{\cE(h^0,h): h\le\bar h\text{ is similar to } h^0\}.
\end{aligned}
\end{equation}
Further, $v$ is extremal in (7.2), resp. (7.3), if and only if $v=W(g)$ with an extremal $g$.
\end{prop}
If $h^1$ is similar to $h^0$, then $\cE(h^1,\cdot)=\cE(h^0,\cdot)+\text{const}$, cf. (4.2). It follows
that the extremal metrics $h$ in (7.2), (7.3) stay the same if $h^0$ is changed to $h^1$.
\begin{defn} %7.4
Extremal hermitian metrics in (7.2), (7.3) will be called minimal, respectively, maximal metrics in the
similarity class of $h^0$.
\end{defn}
If $\bar v$ is bounded, these minimal and maximal hermitian metrics exist according to Proposition 3.1, and if 
the assumptions of Proposition 5.1 are satisfied (with $v_0=\bar v)$, then they are unique. For example,
if $L$ is the hyperplane section bundle $\cO_{\bP_n}(1)$, and $h^0, \bar h$ are smooth and positively curved, 
then by the discussion in section 5 the extremals exist and are unique.
\begin{lem} %7.5
$E\circ W:H^\bC\to\bR$ is a homomorphism to the additive group of real numbers.
\end{lem}
\begin{proof}
If $g_1,g_2\in H^\bC$, then
$e^{-W(g_1g_2)}h^0=g_2^*g_1^*h^0=e^{-g_2^*W(g_1)}g_2^*h^0=e^{-W(g_2)-g_2^*W(g_1)}h^0$, or
\[
W(g_1g_2)=W(g_2)+g_2^*W(g_1). 
\]
Now consider $g\in H^\bC$ and a smooth path $\gamma:[0,1]\to H^\bC$, starting at the neutral element. Let
$\varphi(t)=W\big(\gamma(t)\big)$ and $\psi(t)=W\big(\gamma(t)g\big)=W(g)+g^*W\big(\gamma(t)\big)$.
Thus $\dot\psi(t)=g^*\dot\varphi(t)$ as functions on $M$. We compute
\begin{multline*}
\dfrac d{dt}E\big(W(\gamma(t)g)\big)=
dE\big(\dot\psi(t)\big)=\dfrac 1\vo \int_M\dot\psi(t)\omega_{\psi(t)}^n \\ 
=\dfrac 1\vo \int_M\big(g^*\dot\varphi(t)\big)\big(\gamma(t)g\big)^*\omega^n=
\dfrac 1\vo\int_M \dot\varphi(t)\gamma(t)^*\omega^n=
\dfrac d{dt} E\big(W(\gamma(t))\big).
\end{multline*}
Hence $E\big(W(\gamma(t)g)\big)-E\big(W(\gamma(t))\big)$ is constant. Setting $t=0$ we find
\[
E\big(W(\gamma(t)g)\big)=E\big(W(\gamma(t))\big)+E\big(W(g)\big).
\]
Since $H^\bC$ is connected, this shows that $E\circ W$ is indeed a homomorphism.
\end{proof}

\section{Hermitian metrics on vector bundles} %8

In the Introduction we suggested that one of the motivations
for the variational problems (1.2), (1.3) had to do with hermitian metrics on vector bundles. This comes from a
fundamental construction in geometry that associates with a, say, holomorphic vector bundle 
$F\to B$ a holomorphic line bundle $\one\to\bP F$. Here we use the convention that
$\bP F=\bigcup_{b\in B}\bP F_b$ is the set of one
codimensional subspaces $\Sigma$ in various fibers of $F$, on which the local trivializations of $F$ induce the
structure of a complex manifold fibered over $B$; and $\one$ is a holomorphic line bundle whose fiber over 
$\Sigma\subset F_b$ is $F_b/\Sigma$. Thus $\one$ restricts to $\bP F_b$ as the hyperplane section bundle.
In this section we will write
\[
\pi:\bP F\to B
\]
for the fibering. 

A hermitian metric on a vector space $V$ induces a hermitian metric on any quotient of $V$. Applying this construction
fiberwise, we can associate to a hermitian metric $h_F$ on the vector bundle $F$ a hermitian metric $h_\bP$ on the
line bundle $\one$. We will refer to $h_\bP$ as induced by $h_F$. Note that the restrictions $h_\bP|\onb$ for $b\in B$
are of Fubini--Study type, hence positively curved.
The question is whether there is a natural converse
construction, that associates a hermitian metric on $F$ to a hermitian metric $k$ on $\one$. We will be interested
only in $k$ whose restrictions to $\onb$ are positively curved.

One such construction is as follows. All hermitian metrics on $F$ are induced from one another: If $h^0_F,h_F$ are such, and
$b\in B$, then $h_F|F_b=g^*h^0_F|F_b$ with some $g\in\GL(F_b)$. The corresponding metrics $h^0_\bP, h_\bP$
on $\cO_{\bP F_b}(1)$ are related by the induced action of $g$ on $\cO_{\bP F_b}(1)$. Indeed, if
$\Sigma\subset F_b$ is a hyperplane, the fiber of $\one$ over $\Sigma$, i.e., $F_b/\Sigma$, consists of translates 
$\Sigma'\subset F_b$ of $\Sigma$. Since
\[
h_\bP(\Sigma')=\min_{\zeta\in\Sigma'} h_F(\zeta)=\min_{\zeta\in\Sigma'} h^0_F(g\zeta)=
\min_{\zeta\in g\Sigma'} h^0_F(\zeta)=h^0_\bP(g\Sigma'),
\]
we have
\begin{equation} %8.1
h_\bP|\onb=g^*h^0_\bP|\onb,\qquad g\in\GL(F_b).
\end{equation}

We are in the situation of section 7, with $L=\onb\to M=\bP F_b$, $G=\SU(h^0_F|F_b)$ and $G^\bC=\SL(F_b)$ acting 
on $\bP F_b$, $H^\bC=\GL(F_b)$ acting on $\onb$. (8.1) shows that $h_\bP|\onb$ 
and $h^0_\bP|\onb$ are similar in the sense of Definition 7.2, and clearly the converse is also 
true: Given a metric $h$ on $\one$, if for all $b\in B$ its restrictions to $\onb$ are similar to the restrictions of $h^0_\bP$, 
then it is induced from a hermitian metric $h_F$ on $F$. Moreover, $h_F$ is uniquely determined by $h=h_\bP$: If
$\zeta\in F_b$ and $\Sigma\subset F_b$ is a hyperplane, write $\Sigma ^\zeta\subset F_b$ for the equivalence class of 
$\zeta$ mod $\Sigma$. Then $h_\bP(\Sigma^\zeta)\le h_F(\zeta)$ for all $\Sigma\in\bP F_b$; but the two are equal if
$\Sigma$ is perpendicular (measured in $h_F$) to $\zeta$. Therefore
\[
h_F(\zeta)=\max \{h_\bP(\Sigma^\zeta):\Sigma\in\bP F_b\}.
\]

Hence, as established in section 7, given a smooth hermitian metric $k$ on $\one$, whose 
restrictions to $\onb$ are positively curved, in the
similarity class of induced metrics on each $\onb$ we can find the unique maximal metric $k_b$ among metrics  
$\le k|\onb$, cf. (7.3), Definition 7.4. The collection of $k_b$ defines a metric on
$\one$, which is induced by a hermitian metric on $F$. We associate this hermitian metric on $F$ with $k$, and denote it
$k^F$. Theorem 6.1 implies that $k^F$ is continuous. However, there is no reason why it should be smooth.
To obtain out of a smooth $k$ a smooth metric on $F$, one may have to regularize $k^F$, for which there are
efficient techniques.

This association $k\mapsto k^F$ is based on the variational problems (7.3). It would be equally possible
to construct hermitian metrics on $F$ based on the variational problems (7.2), but those metrics 
seem to be of less interest.\footnote{Naumann in \cite{N21} proposes yet another way to associate a hermitian
metric on $F$ with a hermitian metric on $\one$, using a K\"ahler--Ricci flow on the fibers $\bP F_b$. 
It is not clear if the two constructs are related in some way.}

One situation when a smooth metric on $F$ is desired occurs in connection with Griffiths' conjecture on the
equivalence between amplitude of a holomorphic vector bundle $F\to B$ and the
existence of a hermitian metric on $F$ whose curvature is Griffiths positive ($B$ is compact). By \cite{G69}, amplitude is
equivalent to the existence of a positively curved smooth metric $k$ on $\one$. One can then ask whether the
metric $k^F$ constructed above (perhaps regularized) would be positively curved. In section 10 we will see that
this is not always true, but $k^F$ does have a positivity property:
\begin{thm} %8.1
Let $F\to B$ be a holomorphic vector bundle of rank $m$, $k$ a smooth hermitian metric on the line bundle
$\one$, positively curved on each $\onb$, $b\in B$, and $k^F$ the associated hermitian metric on $F$.
If $\kappa$ is a continuous real $(1,1)$ form on $B$ such that $i\Theta(k)\ge \pi^*\kappa$, then the metric
$\det k^F$ on the line bundle $\det F\to B$ satisfies 
\begin{equation} %8.2
i\Theta(\det k^F)\ge m\kappa.
\end{equation}
In particular, if $k$ is positively curved, then so is $\det k^F$.
\end{thm}

Again, in general $\det k^F$ is only continuous, and (8.2) is a comparison of currents.

Theorem 8.1 will follow from two results. The first expresses $\det k^F$, and in general
$\det h_F$, through Monge--Amp\`ere energy;
the second, for general $(M,\omega)$, establishes that---loosely stated---if the barrier $v_0$ in (1.3) depends on complex
parameters in a plurisubharmonic way, then so does minimal Monge--Amp\`ere energy.
\begin{lem} %8.2
Let $h^0_F$ and $h_F$ be hermitian metrics on $F\to B$, $h^0_\bP, h_\bP$ the induced metrics on the
line bundle $\one$. Define $u:\bP F\to\bR$ by $h_\bP=e^{-u}h^0_\bP$. With the K\"ahler forms
$\omega_b=i\Theta(h^0_\bP|\bP F_b)$, $b\in B$, the Monge--Amp\`ere energy of $u_b=u|\bP F_b\in\cH(\omega_b)$
satisfies
\begin{equation} %8.3
\log\dfrac{\det h_F}{\det h^0_F}(b)=-mE(u_b).
\end{equation}
\end{lem}
\begin{proof}
Since this is a fiberwise statement, it suffices to prove when $B=\{b\}$ is a singleton, and so $F$ is a vector
space. As there is only one fiber, we will omit the $b$'s in our formulas. We can take $F=\bC^m$,
$m\ge 2$, and $h^0_F(\zeta)=\zeta^\dag\zeta$, $\zeta\in\bC^m$ (a column vector). There is a $g\in\GL(m,\bC)$ such that
\begin{equation} %8.4
h_F(\zeta)=(g\zeta)^\dag(g\zeta).
\end{equation}
We will use the same symbol for the corresponding hermitian form $h_F(\zeta,\zeta')=(g\zeta')^\dag(g\zeta)$.
Since $h_\bP=g^* h^0_\bP$ by (8.1), $u_b$ in (8.3) is just $W(g)$ of (7.1). 
If $e_1,\dots,e_m\in\bC^m$ is the standard basis, by definition
\begin{equation*}
(\det h_F)(e_1\wedge\dots\wedge e_m)= \det\big(h_F(e_i,e_j)\big)_{i,j}=\det(e_j^\dag g^\dag g e_i)_{i,j}
=\det g^\dag g.
\end{equation*}
This holds when $g=I$ as well, hence
\begin{equation} %8.5
\det h_F/\det h^0_F=\det g^\dag g=|\det g|^2.
\end{equation}

To see how this is related to the right hand side of (8.3), we vary $g\in\GL(m,\bC)$, and with it, $h_F$. By Lemma 7.5
$E\circ W:\GL(m,\bC)\to\bR$ is a homomorphism of Lie groups. Since $\SL(m,\bC)$ is simple, $E\circ W$ vanishes
on it. If $\lambda\in\bC^\times$ and $g=\lambda I\in\GL(m,\bC)$, 
then $h_F=|\lambda|^2h^0_F$ and $h_\bP=|\lambda|^2h^0_\bP$,
so $W(\lambda I)=-\log|\lambda|^2$ and $E\big(W(\lambda I)\big)=-\log|\lambda|^2$. Writing a general 
$g\in\GL(m,\bC)$ as $\lambda Ig_0$, where $\lambda^m=\det g$ and $g_0\in\SL(m,\bC)$, we obtain
\begin{equation} %8.6
E(u_b)=E\big(W(g)\big)=E\big(W(\lambda I)\big)+ E\big(W(g_0)\big)=-\frac1m\log|\det g|^2.
\end{equation}
Comparison with (8.5) proves (8.3).
\end{proof}
\begin{proof}[Proof of Theorem 8.1] Rather than giving a complete proof, here we will reduce the theorem to Proposition
8.3 below, that will be derived from a more general result in section 9.

Since curvature is computed locally, we can assume that $F$ is trivial, $F=B\times\bC^m\to B$, and $B$ is an open 
subset of some $\bC^d$. We will denote points of $\bC^d$ by $z=(z_1,\dots,z_d)$. It will suffice to prove with any $\var>0$ 
that $i\Theta(\det k^F)\ge m(\kappa-i\var\ddb|z|^2)$ in some neighborhood of an arbitrary $b\in B$. Suppose
$\kappa=i\sum_{\alpha,\beta}\kappa_{\alpha\beta}dz_\alpha\wedge d\bar z_\beta$, and let 
$q(z)=\sum_{\alpha,\beta}\kappa_{\alpha\beta}(b)z_\alpha\bar z_\beta-\var|z|^2/2$. Then
$i\ddb q=\kappa-i\var\ddb|z|^2/2$ at $b$, and therefore close to $b$
\[
\kappa-i\var\ddb|z|^2\le i\ddb q\le\kappa.
\]
Hence it will suffice to prove that $i\Theta(k)\ge\pi^*i\ddb q$ implies $i\Theta(\det k^F)\ge i\,m\ddb q$, the special case
of the theorem when $\kappa=i\ddb q$ with a smooth function $q$.

In this case, along with $k$ consider the metric $e^{\pi^*q}k$ on $\one$ and the metric $(e^{\pi^*q}k)^F=e^q k^F$ on
$F$ associated with it.
Since
\[
\Theta(e^{\pi^*q} k)=\Theta(k)-\pi^*\kappa\qquad\text{and}\qquad 
\Theta\big(\det(e^q k^F)\big)=\Theta(\det k^F)-m\kappa,
\]
$\Theta(k)\ge\pi^*\kappa$ precisely when $\Theta(e^{\pi^*q} k)\ge 0$, and $\Theta(\det k^F)\ge m\kappa$  precisely 
when $\Theta\big(\det(e^qk^F)\big)\ge 0$. This shows that it will be enough to prove when $\kappa=0$.

For the last reduction step let $p:\bP F=B\times \bP_{m-1}\to \bP_{m-1}$ denote the projection, so that $\one$ is the 
pullback of $\cO_{\bP_{m-1}}(1)$ by $p$. Define a reference metric $h^0_F$ by $h_F^0(b,\zeta)=|\zeta|^2$, 
$(b,\zeta)\in B\times\bC^m=F$.  
The induced metric $h^0_\bP$ on $\one$ is the pullback of the Fubini--Study metric on 
$\cO_{\bP_{m-1}}(1)$.  Denoting the curvature of the latter by $-i\omega$, therefore $i\Theta(h^0_\bP)=p^*\omega$.

If $\bar v\in C^\infty(\bP F)$ is defined by $k=e^{-\bar v}h^0_\bP$, we have 
$$
0\le i\Theta(k)=i\Theta(h^0_\bP)+i\ddb \bar v=p^*\omega+i\ddb\bar v.
$$
Since $k$ restricted to $\onb$ is positively curved, $\bar v(b,\cdot)\in\cH(\omega)$. Denote Monge--Amp\`ere energy in 
$\cH(\omega)$ by $E$, and let
\begin{equation} %8.7 
\chi(b)=\min\{E(v):v\in \cH(\omega) \text{ is admissible, }v\ge\bar v(b,\cdot)\},\qquad b\in B.
\end{equation}
By (7.3) and by Lemma 8.2 
\(
\log(\det k^F/\det h^0_F)=-m\chi
\).
Since $h^0_F$ and $\det h^0_F$ are flat, $i\Theta(\det k^F)=i\,m\ddb\chi$ follows. Hence the next proposition implies
Theorem 8.1:
\end{proof}
\begin{prop} %8.3
If $\bar v\in C^\infty(B\times \bP_{m-1})$ satisfies $p^*\omega+i\ddb \bar v\ge0$, 
%and $v(b,\cdot)\in\cH(\omega)$ for all $b\in B$,
then $\chi:B\to\bR$ defined in (8.7) is plurisubharmonic.

\end{prop}
The proof will be given in the next section.

\section{Plurisubharmonic variation} %9

In this section we prove a generalization of Proposition 8.3. Let us return to our connected compact K\"ahler manifold
$(M,\omega)$, on which a connected compact Lie group $G$ acts by holomorphic isometries, with a moment map
$\mu:M\to\fg^*$ and complexified action $G^\bC\times M\to M$. In addition, let $B$ be a complex manifold,
$p:B\times M\to M$ the projection.
\begin{thm} %9.1
If $\bar v:B\times M\to[-\infty,\infty)$ is $p^*\omega$--plurisubharmonic, then
\begin{equation} %9.1
\chi(b)=\inf\{E(v): v\in\cH\text{ is admissible}, v\ge\bar v(b,\cdot)\}
\end{equation}
defines a plurisubharmonic function $\chi:B\to[-\infty,\infty)$.
\end{thm}
In particular, Proposition 8.3 holds.---The proof will need some preparation. We start with a variant of Kiselman's minimum principle \cite{K78}.
\begin{lem} %9.2
Suppose $w$ is a plurisubharmonic function on $B\times G^\bC$, left invariant under $G$ in the sense that
$w(b,g_0g)=w(b,g)$ if $b\in B$, $g_0\in G$, $g\in G^\bC$. Then the function
\[
\chi:B\ni b\mapsto\inf\{w(b,g): g\in G^\bC\}\in[-\infty,\infty)
\]
is plurisubharmonic. The same holds if instead of left, we assume right invariance: $w(b,gg_0)=w(b,g)$ if $g_0\in G$.
\end{lem}
\begin{proof}
Chafi and Loeb in \cite{C01, L85} already proved this under the extra condition that $w$ is smooth, strictly
plurisubharmonic, and exhaustive. Deng, Zhang, and Zhou in \cite{D+14} gave a new proof and a substantial 
generalization; they also sketched an argument to reduce general functions $w$ to  those satisfying the extra conditions.
Since in the setup of our lemma this reduction greatly simplifies, we explain how to do it. We assume $w$ is left
invariant.

Our starting point is \cite[Th\'eor\`eme 2.1]{C01} that implies our claim under the extra assumptions, and we will
show how to get rid of the extra assumptions in two steps. Of course, it suffices to show that $\chi$ is
plurisubharmonic in a neighborhood of an arbitrary $b\in B$. Therefore we can assume $B$ is a pseudoconvex
open subset of some $\bC^k$.

Suppose $w$ is just smooth. Since $G^\bC$ is Stein (Theorem 2.9), $B\times G^\bC$ admits a smooth
strictly plurisubharmonic exhaustion function $f:B\times G^\bC\to[0,\infty)$. By averaging, we can arrange that
$f(b,gh)=f(b,h)$ if $g\in G$. Postcomposing with a suitable increasing, convex, and smooth function $\bR\to\bR$
we can further arrange that $w+\var f$,  for every $\var>0$, is exhaustive and strictly plurisubharmonic. By
\cite{C01}, then
\[
\chi^\var(b)=\inf\{w(b,g)+\var f(b,g): g\in G^\bC\}
\]
is plurisubharmonic. Since $\chi$ is the decreasing limit of $\chi^\var$ as $\var\to 0$, our lemma holds under the
sole supplementary assumption that $w$ is smooth.

Now take a general $w$. Consider the complex Lie group $H=\bC^k\times G^\bC$, with Lie algebra $\fh$. Let
$|\,\,|$ be a Euclidean norm and $\lambda$ a translation invariant measure on $\fh$. Choose 
a compactly supported smooth function $\rho:\fh\to[0,\infty)$ such that $\rho(X)$ depends only on $|X|$, and
$\int_\fh\rho\,d\lambda=1$. If $\var>0$ put $\rho_\var(X)=\rho(X/\var)/\var^{\dim_\bR\fh}$, and define
\begin{equation} %9.2
w_\var(h)=\int_{\supp \rho_\var} w(h\exp X)\rho_\var(X)\,d\lambda(X).
\end{equation}
The domains of $w_\var$ are open sets of form $B_\var\times G^\bC\subset H$ that increase to $B\times G^\bC$ as
$\var\to 0$; $w_\var$ is clearly plurisubharmonic and left invariant under $G$. 
Choose a neighborhood $D\subset \fh$ of $0$
on which $\exp$ is diffeomorphic, and let $\log:\exp D\to D$ denote its inverse. If $h\in H$,
pull back $\lambda$ to a measure $\lambda_h$ on $h\exp D$ by the map $t\mapsto \log(h^{-1}t)$. This is 
absolutely continuous with respect to a fixed smooth measure $\nu$ on $H$, and the Radon--Nikod\'ym
derivative $(d\lambda_h/d\nu)(t)$ is smooth, as a function of $(t,h)$, $h^{-1}t\in \exp D$. If $\var$ is so small
that $\supp\rho_\var\subset D$, by a change of variables
\[
w_\var(h)=\int_{h\exp D} w(t)\rho_\var(\log h^{-1}t)\,d\lambda_h(t).
\]
This shows that $w_\var$ is smooth where defined. Finally, for fixed $h$ (9.2) can be seen as a Euclidean
convolution of two functions on $\fh$, $w(h\exp\cdot)$ and $\rho_\var$, evaluated at $0\in\fh$. As $w(h\exp\cdot)$
is plurisubharmonic, in particular subharmonic, $w_\var(h)$ decreases to $w(h)$ as $\var$ decreases to $0$,
see e.g. \cite[Chapter I (4.18)]{D12}. By what we have already proved, 
\[
\chi_\var(b)=\inf\{w_\var(b,g):g\in G^\bC\}
\]
 is plurisubharmonic, therefore so is the decreasing limit $\chi=\lim_{\var\to0}\chi_\var$.
\end{proof}

We will need to compare minimal Monge--Amp\`ere energy in $\cH=\cH(\omega)$ with that in $\cH(\lambda\omega)$,
$\lambda\in(0,\infty)$. Let $E_\lambda$ be Monge--Amp\`ere energy in $\cH(\lambda\omega)$ and, for
$v_0:M\to[-\infty,\infty)$ bounded above,
\[
I_\lambda(v_0)=\inf\{E_\lambda(v):v\in\cH(\lambda\omega) \text{ is admissible, }v\ge v_0\}\in[-\infty,\infty).
\]
Thus $I$ of section 6 is $I_1$. The following is straightforward to check:
\begin{prop} %9.3
The map $\cH(\omega)\ni v\mapsto\lambda v\in\cH(\lambda\omega)$ is a bijection; it is also a bijection between
admissible elements of $\cH(\omega)$, $\cH(\lambda\omega)$. Further, $E_\lambda(\lambda v)=\lambda E(v)$ and
$I_\lambda(\lambda v_0)=\lambda I(v_0)$.
\end{prop}

In section 3 we have already introduced a smooth $U:G^\bC\to\cH$ such that
$u=U(g)$ solves $i\ddb u=g^*\omega-\omega$, or $\omega_u=g^*\omega$. The normalization for $U$
was $\int_M U(g)\omega^n=0$. For our present purposes  it is better to replace $U(g)$ by $V(g)=U(g)-E\big(U(g)\big)$;
it is still true that $V$ is smooth and $\omega_{V(g)}=g^*\omega$, but now $E\big(V(g)\big)=0$. 
A potential $v\in\cH$ is admissible if and only if $v=V(g)+c$ with some $g\in G^\bC$ and $c\in\bR$, and $v\ge v_0$ 
if and only if $c\ge\max_M \big(v_0-V(g)\big)$. Thus $\chi$
of (9.1) is given by
\begin{equation}\begin{aligned} %9.3
\chi(b)&=\inf\{E(v):v\ge \bar v(b,\cdot) \text{ is admissible}\}\\
&=\inf_{g\in G^\bC}E\big(V(g)+\max_M(\bar v(b,\cdot)-V(g))\big)
=\inf_{g\in G^\bC}\max_{x\in M} \big(\bar v(b,x)-V(g)(x)\big).
\end{aligned}\end{equation}

\begin{lem} %9.4
$V(h)=V(gh)$ if $g\in G$, $h\in G^\bC$, and
\begin{equation} %9.4
V(h)-V(gh)+h^*V(g)=0\qquad\text{if }g,h\in G^\bC.
\end{equation}
\end{lem}

(9.4) means that $V:G^\bC\to C^\infty(M)$ is a cocycle for the group cohomology of $G^\bC$ with 
values in the right $G^\bC$ module $C^\infty(M)$ (the action of $G^\bC$ on $C^\infty(M)$ is by pullback).
\begin{proof}
When $g\in G$, $V(g)=0$, so the first statement follows from (9.4). In turn, (9.4) follows from a
result of Darvas and Rubinstein. In \cite[Lemma 5.8]{DR17} they show---with different notation---that 
\[
g. u=V(g)+g^*u,\qquad g\in G^\bC, \, u\in\cH, \,E(u)=0,
\]
defines a right action of $G^\bC$. Writing out $(gh). u=h. (g. u)$ one obtains (9.4).
\end{proof}
\begin{cor} %9.5
$V(g^{-1})=-(g^{-1})^*V(g)$ if $g\in G^\bC$.
\end{cor}
This follows by putting $h=g^{-1}$ in (9.4).

$V(g)(x)$ was defined through its $\ddb$ with respect to
$x\in M$. It turns out that there is a simple formula for its $\ddb$ with respect to $g\in G^\bC$ as well.
Consider for $x\in M$ the orbit map
\[
o_x:G^\bC\ni g\mapsto gx\in M.
\]
\begin{prop} %9.6
For fixed $x\in M$, $i\ddb V(\cdot)(x)=o_x^*\omega$.
\end{prop}
\begin{proof}
Below we will write $(\ddb)_G$, $(\ddb)_ M$ to indicate operators acting on the $G$ or $M$
variables of functions that are defined on product spaces.---We need to show that if $\xi\in T^{10}G$ then the two forms in 
question, evaluated on $\xi,\bar\xi$, 
give the same result; or that with any $g\in G^\bC$, $X\in\fg^\bC$, and $\phi(s)=g\exp sX$ ($s\in\bC$), 
abbreviating $\partial_s=\partial/\partial s$, $\partial_{\bar s}=\partial/\partial\bar s$,
\begin{align*}
P(s,x)=&\big(i(\ddb)_G V(\cdot)(x)\big)\big(\phi(s)_*\partial_s\, ,\,\phi(s)_*\partial_{\bar s}\big)
=i\partial_s\partial_{\bar s}V\big(\phi(s)\big)(x)
\qquad\text{and}\\
Q(s,x)=&\omega\big(o_{x*}\phi(s)_*\partial_s\,,\,o_{x*}\phi(s)_*\partial_{\bar s}\big)
\end{align*}
are equal.

First we compute $(\ddb)_M P$ and $(\ddb)_M Q$. If locally $\omega=i\ddb w$, then
$Q(s,x)=i\partial_s\partial_{\bar s} w\big(\phi(s)x\big)$,
\begin{align*}
(\ddb)_M Q(s,\cdot)&=i(\ddb)_M\partial_s\partial_{\bar s}w\big(\phi(s)\cdot\big)=
i\partial_s\partial_{\bar s}(\ddb)_M \big(w(\phi(s)\cdot)\big),\\
(\ddb)_M P(s,\cdot)&=i\partial_s\partial_{\bar s}(\ddb)_M V\big(\phi(s)\big)=
\partial_s\partial_{\bar s}\big(\phi(s)^*\omega-\omega\big)=i\partial_s\partial_{\bar s}(\ddb)_M \big( w(\phi(s)\cdot)\big).
\end{align*}
Thus $\ddb P(s,\cdot)=\ddb Q(s,\cdot)$, whence $P(s,\cdot)-Q(s,\cdot)=c$ is constant. Now $X_M$ vanishes
at some $y\in M$ according to Lemma 2.10. For such $y$, $\phi(s)(y)$ is independent of $s$, and
$Q(s,y)=0$.  All we need to show is that $P(s,y)=0$ as well: this will imply $c=0$, and with it,  the proposition.

We compute $P(s,y)$ as follows. Since $(\exp sX)y=y$, Lemma 9.4 gives $V(g\exp sX)(y)=V(\exp sX)(y)+V(g)(y)$.
Substituting $g\exp\sigma X$ for $g$,
\begin{gather} %9.5
V\big(g\exp(\sigma+s)X\big)(y)=
V(\exp \sigma X)(y)+V(\exp sX)(y)+V(g)(y)\quad\text{and} \\ %9.6
\partial_s\partial_{\bar s} V\big(g\exp(\sigma+s)X\big)(y)=\partial_s\partial_{\bar s}V(\exp sX)(y)
\end{gather}
follow. Setting $\sigma=-s$ in (9.5), and taking $\partial_s\partial_{\bar s}$,
we obtain
\[
0=\partial_s\partial_{\bar s}V(\exp(-sX))(y)+\partial_s\partial_{\bar s}V(\exp sX)(y).
\]
This implies $\partial_s\partial_{\bar s}|_{s=0}V(\exp sX)(y)=0$. Letting $s=0$ in (9.6) then gives
$P(\sigma,y)=i\partial_\sigma\partial_{\bar\sigma} V(g\exp\sigma X)(y)=0$, as needed.
\end{proof}
\begin{proof}[Proof of Theorem 9.1]
First we assume that $\bar v$ is smooth).
We start with representation (9.3) of $\chi$, but replace $g$ by $g^{-1}$. Thus 
$\chi(b)=\inf_{g\in G^\bC}w(b,g)$, where
\begin{equation}\begin{aligned} %9.7
w(b,g)=&\max_{x\in M}\big(\bar v(b,x)-V(g^{-1})(x)\big)=\max_{x\in M}\big(\bar v(b,x)+V(g)(g^{-1}x)\big)\\
=&\max_{x\in M}\big(\bar v(b,gx)+V(g)(x)\big)
\end{aligned}\end{equation}
(here we used Corollary 9.5). Let $\pi_G:B\times G\to G$ be the projection.
With $x\in M$ fixed, consider functions $\bar u,u:B\times G\to\bR$,
\[
\bar u(b,g)=\bar v(b,gx),\qquad
u(b,g)=\bar u(b,g)+V(g)(x).
\]
If $(\id _B\times o_x)(b,g) $ is defined as $\big(b,o_x(g)\big)$, we can write $\bar u=(\id_B\times o_x)^*\bar v$. 
By assumption
\[
i\ddb\bar u=i(\id_B\times o_x)^*\ddb\bar v\ge -(\id_B\times o_x)^*p^*\omega=-\pi_G^*o_x^*\omega,
\]
and $i\ddb u\ge -\pi_G^*o_x^*\omega+\pi_G^*i\ddb V(\cdot)(x)=0$ by 
Proposition 9.6. In other words, $u$ is plurisubharmonic, and so the continuous function
$w$, the maximum of plurisubharmonic functions, is itself
plurisubharmonic. But the second expression in (9.7), in conjunction with Lemma 9.4, shows that $w$ is right
invariant, $w(b,g)=w(b,gg_0)$ if $g_0\in G$. Therefore by Lemma 9.2 
$\chi(b)=\inf_{g\in G^\bC} w(b,g)$ is plurisubharmonic.

Next consider a $p^*\omega$--plurisubharmonic $\bar v$ that is just bounded. We can assume $B\subset\bC^k$
is open and bounded. Denoting points of $\bC^k$ by $z$ and letting $q(z,x)=|z|^2$, $(z,x)\in B\times M$, the form
$\theta=i\ddb q+ p^*\omega$ is positive on $B\times M$. Regularization techniques of Demailly and P\v aun
\cite{D92, DP04}, but especially \cite[Theorem 2]{BK07} of B\l ocki and Ko\l odziej imply that there is a
sequence of smooth $(\theta/j+ p^*\omega)$--plurisubharmonic functions $u_j$ decreasing to $\bar v$. (To apply
the B\l ocki--Ko\l odziej result we may need to shrink $B$.) 
The functions $v_j(z,x)=u_j(z,x)+|z|^2/j$ are then
$(1+1/j)p^*\omega$--plurisubharmonic, and also decrease to $\bar v$. At the price of adding a constant, we can assume
that $\bar v$ is negative; then $u_j, v_j$ can be made negative as well, for large $j$. By what we 
have already proved
\[
\chi_j(b)=\inf\big\{E_{1+1/j}(v): v\in\cH\big((1+1/j)\omega\big)\text{ is admissible, }v\ge v_j(b,\cdot)\big\}
\]
is plurisubharmonic, and eventually $\le 0$. Now
\[
\chi_j(b)=I_{1+1/j}\big(v_j(b,\cdot)\big)=\frac {j+1}jI\bigg(\dfrac {jv_j\big(b,\cdot)}{j+1}\bigg)
\]
by Proposition 9.3. Since the $v_j$ are (eventually) negative, $jv_j/(j+1)$ decrease to $\bar v$ as $j\to\infty$. In light of
Proposition 6.2 $\chi(b)=I\big(\bar v(b,\cdot)\big)$ is the decreasing limit of $j\chi_j(b)/(j+1)$, therefore plurisubharmonic.

Finally, consider a general $\bar v$. With $j\in\bN$, set this time $v_j=\max(\bar v,-j)$. By what we have already
proved, $\chi_j(b)=I\big(v_j(b,\cdot)\big)$ are plurisubharmonic functions of $b$, hence so is their, again by Proposition
6.2 decreasing, limit $I\big(\bar v(b,\cdot)\big)=\chi(b)$.
\end{proof}

To conclude this section we digress to describe how the function $V:G\to\cH$ gives rise to solutions of the
Monge--Amp\`ere equation. While this is not needed for the main purpose of the paper, it generalizes Mabuchi's
\cite[Theorem 3.5]{M87}, that we used in the proof of Proposition 2.5, and provides an independent proof.
Let $\pi:\bC\times M\to M$ be the projection and, as before, $n=\dim_\bC M$.
\begin{thm} %9.7
If $g\in G^\bC$ and $Z\in\fg^\bC$ then $\Phi(s,x)=V\big((\exp sZ)g\big)(x)$, $s\in\bC$, $x\in M$, satisfies
\begin{equation} % 9.8
(\pi^*\omega+i\ddb\Phi)^{n+1}=0.
\end{equation}
\end{thm}

In particular, let $g$ be the neutral element, $X\in\fg$, and $Z=iX$. Writing $s=\sigma+it$,
\[
V(\exp sZ)=V\big(\exp(-t X)\exp(i\sigma X)\big)=V(\exp i\sigma X)
\]
by Lemma 9.4, and so $\Phi(s,x)=\Phi(\re s,x)$. Hence (9.8) implies that $\sigma\mapsto\Phi(\sigma,\cdot)$ is a 
geodesic, cf. the discussion around
(2.3); this is equivalent to \cite[Theorem 3.5]{M87}.
\begin{proof}
Any $\bR$-linear form on $\fg$ extends to a unique $\bC$-linear form on $\fg^\bC$. Thus $\fg^*$ embeds
into $(\fg^\bC)^*$, and accordingly, the moment map $\mu$ can be viewed as a map $\mu:M\to(\fg^\bC)^*$. Let
$g_s=\exp sZ$. Let $Z=X+iY$ with $X,Y\in\fg$, and write $\bar Z=X-iY$. 

To compute $\pi^*\omega+i\ddb\Phi$ we start by recalling
\begin{equation} %9.9
\omega+i\ddb V(g_sg)=(g_sg)^*\omega=g^*g_s^*\omega.
\end{equation}
Therefore, with $\sigma\in\bR$,
\[
i\ddb\,\partial_\sigma|_{\sigma=0}V(g_\sigma g)=
\partial_\sigma|_{\sigma=0}\,g^*(\exp i\sigma Y)^*(\exp\sigma X)^*\omega 
=g^*\cL_{JY_M}\omega=-2ig^*\ddb\langle\mu,Y\rangle
\]
by Proposition 2.8. Hence $\partial_\sigma|_{\sigma=0}V(g_\sigma g)=-2g^*\langle\mu,Y\rangle+c$,
where $c$ is a constant function on $M$. With $s=\sigma+it$ it follows upon replacing $g$ by $g_{s_0}g$
\[
\partial_\sigma V(g_s g)=-2g^*g_s^* \langle\mu,Y\rangle+c,\qquad
\partial_t V(g_s g)=-2g^*g_s^*\langle\mu,X\rangle+c',
\]
and with $\partial_s=(\partial_\sigma-i\partial_t)/2$, $\partial_{\bar s}=(\partial_\sigma+i\partial_t)/2$,
\begin{equation} %9.10
\partial_s V(g_s g)=ig^*g_s^*\langle\mu, Z\rangle+c'',\qquad
\partial_{\bar s} V(g_s g)=-ig^*g_s^*\langle\mu,\bar Z\rangle+c'''.
\end{equation}
If $Z\in\fg$ then $d\langle\mu, Z\rangle=-\iota(Z_M)\omega=-\iota(Z_M')\omega-\iota(Z_M'')\omega$, whence
$\bar\partial\langle\mu,Z\rangle=-\iota(Z_M')\omega$. By $\bC$-linearity, the same holds for $Z\in\fg^\bC$, and
passing to conjugates $\partial\langle\mu,\bar Z\rangle=-\iota(Z_M'')\omega$ follows.
Therefore (9.10) implies
\begin{equation} %9.11
\begin{aligned}
\partial_s\bar\partial V(g_s g)&=\sqrt{-1}g^*g_s^*\bar\partial\langle\mu, Z\rangle=-\sqrt{-1}g^*g_s^*\iota(Z_M')\omega,\\
\partial_{\bar s}\partial V(g_s g)&=-\sqrt{-1}g^*g_s^*\partial\langle\mu, \bar Z\rangle=\sqrt{-1}g^*g_s^*\iota(Z_M'')\omega.
\end{aligned}
\end{equation}

Since we can change $g$ to $g_{s_0}g$, it suffices to prove (9.8) at points $(0,x)\in\bC\times M$. Choose a 
smooth function $w$ on a neighborhood of a fixed $x_0\in M$ such that $\omega=i\ddb w$ there. By Proposition 9.6
$i\ddb V(\cdot)(x)=o_x^*\omega$, hence near $x_0$, when $s=0$
\begin{equation} %9.12
\partial_s\partial_{\bar s}V(g_s g)=g^*\partial_s\partial_{\bar s}(\exp sZ)^*w=
-\sqrt{-1}g^*\iota(Z_M')\iota(Z_M'')\omega.
\end{equation}
Extend the action of $g$ to $\bC\times M$, by letting it act as the identity on $\bC$. 
(9.9), (9.11), and (9.12) give (still when $s=0$)
\[
(g^*)^{-1}(\pi^*\omega+\sqrt{-1}\ddb\Phi)=
\omega+ds\wedge\iota(Z_M')\omega+\iota(Z_M'')\omega\wedge d\bar s
+\iota(Z_M')\iota(Z_M'')\omega\, ds\wedge d\bar s.
\]
(On the right  $ds, d\bar s$, and $\omega$ are pulled back to $\bC\times M$.) One checks that
$\partial_s-Z_M'$ is in the kernel of the above $2$-form. This means that $\text{rk}\,\pi^*\omega+i\ddb\Phi\le n$, 
which proves the theorem.
\end{proof}
\section{An example} %10

Let $F\to B$ be a holomorphic vector bundle, $\one\to\bP F$ the associated line bundle, $k$ a smooth
hermitian metric on $\one$, and $k^F$ the associated hermitian metric on $F$, as in section 8. The question we
entertain here is how the association $k\mapsto k^F$ impacts curvature: If $k$ is positively curved, must
$k^F$ also be positively, or at least semipositively, curved? According to Griffiths \cite[(2.36), (2.37)]{G69},
a smooth hermitian metric $h_F$ on $F$ is (semi)positively curved if and only if the induced metric $h_\bP$ on $\one$  is.
For this reason the question above, in its weaker form, can be rephrased as follows: Consider a complex
manifold $B$ and the product $B\times\bP_n$, $n\in\bN$, with projection $p:B\times\bP_n\to\bP_n$. Let
$\omega$ be the Fubini--Study K\"ahler form on $\bP_n$, and endow $\bP_n$ with the standard actions of
$G=\SU(n+1)$, $G^\bC=\SL(n+1,\bC)$.

\phantom{mmn} Suppose $\bar v\in C^\infty(B\times \bP_n)$ satisfies $p^*\omega +i\ddb \bar v>0$. Let
$u:B\times\bP_n\to\bR$ be such
\newline (10.1)\phantom{m} that for each
$b\in B$, $u(b,\cdot)\in\cH(\omega)$ minimizes Monge--Amp\`ere energy $E(v)$ among 
\newline\phantom{mmmm} admissible $v\ge\bar v(b,\cdot)$. Is  $u$ $p^*\omega$-plurisubharmonic on $B\times\bP_n$?

Indeed, with $F=B\times \bC^{n+1}\to B$ the trivial bundle, $\one$ is the pullback of the hyperplane section
bundle $\cO_{\bP_n}(1)$ by $p$, and the trivial metric $h^0_F(b,\zeta)=\zeta^\dagger\zeta$ on $F$ induces
a metric $h^0_\bP$ on $\one$, which is just the pullback of the Fubini--Study metric on $\cO_{\bP_n}(1)$. If
$\bar v, u$ are as in (10.1), define metrics $k=e^{-\bar v}h^0_\bP$ and $e^{-u}h^0_\bP$ on $\one$. On the one hand, that
$k$ is positively curved means $i\Theta(k)=p^*\omega+i\ddb\bar v>0$; on the other, $i\Theta(k^F)\ge 0$ if and 
only if the metric induced by $k^F$, namely $e^{-u}h^0_\bP$, satisfies $i\Theta(e^{-u}h^0_\bP)=p^*\omega+i\ddb u\ge 0$.
Therefore an affirmative answer to the original question implies the same for its reformulation (10.1).

However, the answer to (10.1) is ``not always''. Hence Griffiths' conjecture cannot be proved just by applying
the construction $k\mapsto k^F$ to a positively curved metric $k$ on $\one$.
\begin{thm} %10.1
There are an open $B\subset\bC$ and $\bar v\in C^\infty(B\times\bP_1)$ such that $p^*\omega+i\ddb\bar v>0$,
but $u\in C(B\times\bP_1)$ constructed in (10.1) fails to be $p^*\omega$-plurisubharmonic.
\end{thm}
In constructing the example it is easier to start with $u$. Let $A\subset\bC$ be the unit disc. 
Using homogeneous coordinates $(x_0:x_1)$ on $\bP_1$, $u$ will be a perturbation of
\setcounter{equation}{1}
\begin{equation} %10.2
u_0(z,x)=\log\dfrac{e^{2|z|^2}|x_0|^2+e^{-|z|^2}|x_1|^2}{|x_0|^2+|x_1|^2}
=-|z|^2+\log\dfrac{e^{3|z|^2}|x_0|^2+|x_1|^2}{|x_0|^2+|x_1|^2},
\end{equation}
$z\in A$, $x=(x_0:x_1)\in \bP_1$. (The coefficient $2$ in front of $|z|^2$ could be replaced by any number $>1$.)
This is certainly not subharmonic when restricted to $A\times\{(0:1)\}$. The immediate question is, is there a 
strictly $p^*\omega$-plurisubharmonic $\bar v\in C^\infty(A\times\bP_1)$ (i.e., $i\ddb\bar v+p^*\omega>0$) 
for which the construction in 
(10.1) outputs this 
$u=u_0$? Of course, $u_0$ must dominate $\bar v$, but not by much, to ensure it minimizes energy. 
One may try to take $\bar v$ the $p^*\omega$-plurisubharmonic
envelope of $u_0$, the largest $p^*\omega$-plurisubharmonic function $\le u_0$. This will essentially work, except that
such envelopes are rarely smooth or strictly $p^*\omega$-plurisubharmonic. Therefore the true $\bar v$ will  be
a regularization of the envelope.

In what follows, we will construct this envelope, $v_0$---although, in the end we will not need to know that the
$v_0$ we have found
is the envelope.---It is determined by three regimes
\begin{gather*}
D_0=\big\{|x_1|^2<2e^{3|z|^2}|x_0|^2\big\},\qquad D_2=\big\{2e^3|x_0|^2<|x_1|^2\big\}, \\
D_1=\big\{2e^{3|z|^2}|x_0|^2<|x_1|^2<2e^3|x_0|^2\big\},
\end{gather*}
and three functions, in addition to $u_0$,
\begin{equation}\begin{aligned} %10.3
u_1(x)&=-\dfrac13\log\dfrac{|x_1|^2}{2|x_0|^2}+\log\dfrac{3|x_1|^2/2}{|x_0|^2+|x_1|^2},\qquad\text{and}\\
u_2(x)&=-1+\log\dfrac{e^3|x_0|^2+|x_1|^2}{|x_0|^2+|x_1|^2}.
\end{aligned}\end{equation}
In these formulae $x=(x_0:x_1)\in\bP_1$ and $z\in A$. The value of $u_1$ when $x_0$ or $x_1=0$ is computed
as a limit $=-\infty$.
\begin{prop} %10.2
$u_1,u_2$ are $\omega$-subharmonic on $\bP_1$ and $u_0$ is strictly $p^*\omega$-plurisubharmonic on $D_0$.
\end{prop}
\begin{proof}
With the projection $\pi:\bC^2\setminus\{0\}\to\bP_1$, the first question is whether $\pi^*u_1,\pi^*u_2$ are
$\pi^*\omega$-plurisubharmonic on $\bC^2\setminus\{0\}$. Since $\pi^*\omega=i\ddb\log\big(|x_0|^2+|x_1|^2\big)$,
the question is whether the functions
\begin{align*}
&-\dfrac13\log\dfrac{|x_1|^2}{2|x_0|^2}+\log\dfrac{3|x_1|^2}2=\dfrac13\log(2|x_0|^2)+\log\dfrac{3|x_1|^{4/3}}{2}
\quad\text{and}\\
&-1+\log(e^3|x_0|^2+|x_1|^2)
\end{align*}
are plurisubharmonic on $\bC^2\setminus\{0\}$: they obviously are.

To discuss $u_0$, we start with a formula for smooth functions $r,s$:
\[
\ddb\log( e^r+e^s)=
\dfrac{e^r\ddb r+e^s\ddb s}{e^r+e^s}+e^{r+s}\dfrac{\partial(r-s)\wedge\bar\partial(r-s)}{(e^r+e^s)^2}.
\]
Since $x_0\neq 0$ on $D_0$, we can use the local coordinate $\xi=x_1/x_0$. At least when $\xi\neq 0$
\[
u_0(z,x)=-|z|^2+\log\dfrac{e^{3|z|^2}|x_0|^2+|x_1|^2}{|x_0|^2+|x_1|^2}=
-|z|^2+\log(e^r+e^s)-\log(1+|\xi|^2),
\]
where $r=3|z|^2$, $s=\log|\xi|^2$. Hence 
\begin{align*}
p^*\omega+i\ddb u_0&=-idz\wedge d\bar z+\dfrac{3ie^{3|z|^2}\,dz\wedge d\bar z}{e^{3|z|^2}+|\xi|^2}+
ie^{r+s}\dfrac{\partial (r-s)\wedge\bar\partial(r-s)}{(e^r+e^s)^2}
\\
&=\dfrac{i(2e^{3|z|^2}-|\xi|^2)}{e^{3|z|^2}+|\xi|^2}dz\wedge d\bar z+
ie^r\dfrac{\alpha\wedge\bar\alpha}{(e^r+e^s)^2},
\end{align*}
where $\alpha=3\bar z\bar\xi\,dz-d\xi$. By continuity, this formula holds when $\xi=0$, too. Above, the coefficients of 
$i\,dz\wedge d\bar z$ and $i\alpha\wedge\bar\alpha$ are positive in $D_0$, and the forms $dz$ and $\alpha$ are
everywhere linearly independent. It follows that $p^*\omega+i\ddb u_0$ is indeed positive.
\end{proof}
\begin{prop} %10.3
The formula
\begin{equation} %10.4
u_{12}(x)=\begin{cases} u_1(x)\quad\text{if}\quad 2e^3|x_0|^2\le |x_1|^2\\ 
u_2(x)\quad\text{if}\quad 2e^3|x_0|^2\ge |x_1|^2\end{cases}
\end{equation}
defines a continuous $\omega$-subharmonic function $\bP_1\to[-\infty,\infty)$.
\end{prop}
\begin{proof} 
Proposition 10.2 implies that $u_{12}$ is $\omega$-subharmonic away from the circle 
$K=\{x:2e^3|x_0|^2=|x_1|^2\}$. On $K$ itself $u_1,u_2$ agree to second order: writing $t=|x_1/x_0|^2$, near $K$
\[
u_2(x)-u_1(x)=-1+\log(t+e^3)+\dfrac13\log(t/2)-\log(3t/2),
\]
and the function on the right as well as its derivative vanish when $t=2e^3$. In particular, in (10.4) the two 
definitions for $u_{12}(x)$ when $x\in K$ are consistent, and $u_{12}$ is continuous. It also follows that 
 near $K$ $du_{12}$
exists and is Lipschitz continuous locally. Hence near $K$ the distributional Laplacian of $u_{12}$ is locally
$L^\infty$. Since $\omega+i\ddb u_{12}\ge 0$ as a current on $\bP_1\setminus K$, it follows that 
$\omega+i\ddb u_{12}\ge 0$ everywhere, and  $u_{12} $ is indeed $\omega$-subharmonic.
\end{proof}
\begin{prop} %10.4
The formula
\begin{equation} %10.5
v_0(z,x)=\begin{cases}u_0(z,x)&\text{if}\quad |x_1|^2\le 2e^{3|z|^2}|x_0|^2\\
u_{12}(x)&\text{if}\quad |x_1|^2\ge 2e^{3|z|^2}|x_0|^2\end{cases}
\end{equation}
defines a continuous $p^*\omega$-plurisubharmonic function $v_0:A\times\bP_1\to\bR$, and for each
$z\in A$ the function $u_0(z,\cdot)\ge v_0(z,\cdot)$ is the unique minimizer of $E(v)$ among admissible
$v\in\cH(\omega)$, $v\ge v_0(z,\cdot)$.
\end{prop}
\begin{proof}
Fix $x\in\bP_1$. Differentiation shows that as a function of $|z|$, $0\le|z|\le 1$,
\[
e^{2|z|^2}|x_0|^2+e^{-|z|^2}|x_1|^2
\]
has a strict minimum. If $|x_1|^2\le2e^3|x_0|^2$, the minimum occurs when $|x_1|^2=2e^{3|z|^2}|x_0|^2$
(write $|z|_\text{min}$ for this value of $|z|$); if $|x_1|^2\ge2e^3|x_0|^2$, the minimum occurs when $|z|=1$.
It follows that if $|x_1|^2<2e^3|x_0|^2$ then
\[
u_0(z,x)\ge u_0(|z|_\text{min}, x)=u_1(x),
\]
and if $|x_1|^2\ge2e^3|x_0|^2$ then $u_0(z,x)\ge u_0(1,x)=u_2(x)$. From this we first infer that in (10.5) 
the two definitions for
$v_0(z,x)$, when $|x_1|^2=2e^3|x_0|^2$, are consistent, and $v_0\le u_0$ is continuous. The second inference is that
$u_0$ and $u_{12}$ agree to second order at 
\[
\tilde K=\{(z,x)\in A\times\bP_1: |x_1|^2= 2e^{3|z|^2}|x_0|^2\}.
\]
Since $v_0$ is obviously $p^*\omega$-plurisubharmonic on
$(A\times\bP_1)\setminus\tilde K$, it follows---as in the proof of Proposition 10.3---that it is
$p^*\omega$-plurisubharmonic on all of $A\times\bP_1$.

Now fix $z\in A$. To show that $u_0(z,\cdot)$ uniquely minimizes Monge--Amp\`ere energy among admissible 
$v\in\cH(\omega)$, $v\ge v_0(z,\cdot)$, we will apply Proposition 5.2. The diagonal matrix 
$$
\gamma=\text{diag}(e^{3|z|^2/4}, e^{-3|z|^2/4})\in\SL(2,\bC)
$$
acts on $\bP_1$ as the diagonal matrix 
$\text{diag}(e^{|z|^2}, e^{-|z|^2/2})$, and pulls back $\omega=i\ddb\log(|x_0|^2+|x_1|^2)$ to $\omega+i\ddb u_0(z,\cdot)$.
Let $X\in\mathfrak{su}(2)$. It is a matrix of form
\[
X=i\begin{pmatrix} P & Q\\
\bar Q & -P
\end{pmatrix},
\]
with $P\in\bR$, $Q\in\bC$. Using formula (5.3) for the moment map in $\bP_n$,
\begin{equation} %10.6
\big\langle\mu(\gamma(x)),X\big\rangle=
\dfrac{Pe^{3|z|^2/2}|x_0|^2+Q\bar x_0x_1+\bar Q x_0\bar x_1-Pe^{-3|z|^2/2}|x_1|^2}
{e^{3|z|^2/2}|x_0|^2+e^{-3|z|^2/2}|x_1|^2}
\end{equation}
up to a constant factor that will play no role. If we denote by $m_r$ the rotation invariant probability measure on the circle 
$K_r=\{(1:re^{it})\in\bP_1\mid t\in\bR\}$, $r>0$, (10.6) gives
\begin{align*}
\int_{K_r}\langle\mu\circ\gamma,X\rangle\,dm_r=&
\dfrac1{2\pi}\int_0^{2\pi}\dfrac{Pe^{3|z|^2/2}+Qre^{it}+\bar Qre^{-it} -Pr^2e^{-3|z|^2/2}}{e^{3|z|^2/2}+r^2e^{-3|z|^2/2}}\,dt\\
=&\dfrac{Pe^{3|z|^2/2}-Pr^2e^{-3|z|^2/2}}{e^{3|z|^2/2}+r^2e^{-3|z|^2/2}}.
\end{align*}
Thus $\int (\mu\circ\gamma)\,dm_r=0$ if $r=e^{3|z|^2/2}$. By (10.5) $\bar v(z,\cdot)$ in a neighborhood of 
$K_r$ agrees with $u_0(z,\cdot)$, and $u_0(z,\cdot)$ is smooth and strictly $\omega$-subharmonic there
(Proposition 10.2). Hence
Proposition 5.2 gives that $u_0(z,\cdot)$ is indeed the unique minimizer of $E$.
\end{proof}

%At this point it would not be difficult to show that $v_0$ of (10.5) is in fact the largest $p^*\omega$-plurisubharmonic 
%function $\le u_0$; but this will not be needed for the rest.
\begin{proof}[Proof of Theorem 10.1] 
In the setup of Proposition 10.4, let $B\subset A$ be a non--empty, relatively compact open subset. It follows
from \cite[Theorem 2]{BK07} and from Dini's theorem that there is a decreasing sequence of strictly 
$p^*\omega$-plurisubharmonic functions $v^j\in C^\infty(B\times \bP_1)$ converging uniformly to
$v_0|B\times\bP_1$.  By Propositions 5.1, 10.4
$v^j(b,\cdot), v_0(b,\cdot)\in C_\text{unique}$ for each $b\in B$ (see section 6 for this notation).
Define  $u^j:B\times\bP_1\to\bR$ so that $u^j(b,\cdot)\in\cH(\omega)$ minimizes $E(v)$ among admissible $v\ge v^j(b,\cdot)$, 
for each $b\in B$. 
Since $v^j$ decrease as $j$ grows, so do $u^j$. 
Hence Proposition 6.1 implies that the $u^j$ are continuous and decrease to $u_0$ on $B\times\bP_1$, 
locally uniformly by 
Dini's theorem. Not all $u^j$ can be $p^*\omega$-plurisubharmonic, since their limit is not. This means
that one of the $v^j$ will do as $\bar v$ of the theorem.
\end{proof}

\end{document}